\documentclass[11pt,oneside]{amsart}
\usepackage{amsmath,amsthm,amssymb,amscd,fancybox,esint}
\usepackage[a4paper,left=3cm,right=3cm,top=3cm,bottom=3cm]{geometry}
\usepackage{cite}
\allowdisplaybreaks
\renewcommand{\epsilon}{\varepsilon}
\renewcommand{\phi}{\varphi}

 \newcommand{\bZ}{\mathbb{Z}}
\newcommand{\bR}{\mathbb{R}} \newcommand{\bN}{\mathbb{N}}
 
\newcommand{\dS}{\mathcal{S}(\bR^d)} 

\newcommand{\cU}{\mathcal{U}}
\newcommand{\ddS}{\mathcal{S}'(\bR^d)}

\newcommand{\supp}{\operatorname{supp}}

\newcommand{\brac}[1]{\langle #1\rangle}
\newtheorem{theorem}{Theorem}[section]
\newtheorem{lemma}[theorem]{Lemma}
\newtheorem{proposition}[theorem]{Proposition}
\newtheorem{prop}[theorem]{Proposition}
\newtheorem{corollary}[theorem]{Corollary}
\theoremstyle{definition}
\newtheorem{definition}[theorem]{Definition}
\newtheorem{example}[theorem]{Example}
\theoremstyle{remark}
\newtheorem{remark}[theorem]{Remark}
\numberwithin{equation}{section}
\renewcommand{\d}{\,\text{d}}
\def\dx{{\,\operatorname{d}\!x}}
\def\dy{{\,\operatorname{d}\!y}}
\def\dt{{\,\operatorname{d}\!t}}
\def\du{{\,\operatorname{d}\!u}}

\def\bZ{{\mathbb Z}}

\def\bN{{\mathbb N}}

\def\bC{{\mathbb C}}
\def\bR{{\mathbb R}}

\def\bT{{\mathbb T}}

\def\cB{\mathcal{B}}
\def\cC{\mathcal{C}}
\def\cD{{\mathbb{R}^d}}

\def\cU{\mathcal{U}}

\newcommand{\C}{\mathbb{C}}
\newcommand{\R}{\mathbb{R}}
\newcommand{\vf}{\mathbf{f}}
\newcommand{\vg}{\mathbf{g}}

\newcommand{\vx}{\mathbf{x}}
\newcommand{\vc}{\mathbf{c}}
\newcommand{\vd}{\mathbf{d}}

\newcommand{\ve}{\mathbf{e}}
\newcommand{\vs}{\mathbf{s}}
\newcommand{\vw}{\mathbf{w}}

\providecommand{\abs}[2][]{#1\lvert#2#1\rvert}
\providecommand{\norm}[2][]{#1\lVert#2#1\rVert}
\newcommand{\esssup}{\mathop{\rm ess{\,}sup}}
\def\cprime{$'$} 
\begin{document}
\title[Matrix $A_p$-weights relative to a pseudo-metric]{ Matrix $A_p$-weights relative to a pseudo-metric}
\author[M.\ Nielsen]{Morten Nielsen}
\thanks{This work was supported by the Independent Research Fund Denmark, grant no.\ 5281-00046B}
\address{Department of Mathematical Sciences\\ Aalborg
  University\\ Thomas Manns Vej 23\\ DK-9220 Aalborg East\\ Denmark}
\email{mnielsen@math.aau.dk}
\begin{abstract}
 Matrix weights satisfying a Muckenhoupt $A_p$-condition relative to a family of anisotropic balls in $\bR^d$ defined by a pseudo-metric are studied. It is shown that such matrix weights satisfy a doubling condition and a reverse H\"older inequality. In the special case, where the pseudo-metric is homogeneous with respect to a one-parameter dilation group, the corresponding Muckenhoupt class is shown to satisfy an invariance property under composition with affine transformations generated by the dilation group. A general sampling theorem is derived for the matrix-weighted space $L^p(W)$ for Muckenhoupt $A_p$ weights $W$ along with a corresponding multiplier result for $L^p(W)$. An application of the results to the study of anisotropic matrix-weighted Besov spaces is considered.    
\end{abstract}
\maketitle

\section{Introduction}
Weights satisfying a Muckenhoupt $A_p$-condition play an important role in weighted norm theory for the study of maximal operators and various singular integral operators on weighted $L^p$ spaces. For example, the celebrated Muckenhoupt-Hunt-Wheeden theorem \cite{HuntMuckWhee:1973a}, stated in the setting of $\bR^d$, shows that the Riesz transform(s) on $\R^d$ are bounded on weighted $L^p(\bR^d;w) $ precisely when the weight $w$ satisfies the Muckenhoupt $A_p$-condition, i.e.,  
\begin{equation}\label{eq:APclassical}
[w]_{A_p}:=\sup_{B\in\cB}\frac{1}{|B|}\int_B w(x)\dx\cdot\bigg[\frac{1}{|B|}\int_B w^{-\frac{1}{p-1}}(x)\dx\bigg]^{p-1}<\infty,
\end{equation}
where $\cB$ is the collection of all Euclidean balls in $\bR^d$.

More recently there has been significant focus on extending weighted norm theory to the setting of vector-valued functions with weights provided by matrix-valued functions that are positive definite almost everywhere. 
A highlight 
in the matrix-weighted case is the formulation of a suitable matrix $A_p$-condition by Nazarov, Treil and Volberg that completely characterizes boundedness of the Riesz transform(s) on matrix-weighted $L^p$ on $\bR^d$ for $1<p<\infty$, see \cite{Volb:1997a,TreiVolb:1997a,NazaTrei:1996a}. Applications of the matrix $A_p$-conditions to the study of vector-valued function spaces and associated operators have also generated significant recent interest, see, e.g., \cite{Roud:2003a,FrazRoud:2004a, Gold:2003a,Niel:2025a,CruzMoenRodn:2016a,IsraKwonPott:2017a,BuHytoYang:2025a} and references therein.

 The definition of the $A_p$-condition for $d>1$ suggests that the condition must be tied to the specific geometry defined by the sets in $\cB$. The focus in the present paper is on a scenario where we replace $\cB$ by other structured notions of ``balls'' in the context of matrix weights.  

For scalar weights, extensions of the Muckenhoupt condition to covering systems with other geometric characteristics than Euclidean balls have a long history. In a study that will be central for our results, Calder\'on \cite{Cald:1976a} considered a specific family of generalized scalar Muckenhoupt weights, where Euclidean balls are replaced by balls generated by a pseudo-metric in \eqref{eq:APclassical}, see also \cite{Kurt:1975a}. Calder\'on's work also clarifies that changing the structure of $\cB$ may introduce non-trivial mathematical challenges that often require new techniques.

A systematic study of the interplay between the geometry of the covering system, scalar weights, and boundedness properties of corresponding maximal operators was conducted by Jawerth \cite{Jawe:1986a}, see also \cite{Pere:1991b}. For matrix weights much less is known about Muckenhoupt conditions based on other geometries. A first study of the Muckenhoupt condition for matrix weights relative to certain general structured coverings of $\bT^d$ was conducted by \v{S}iki\'c and the author in \cite{NielSiki:2025a}.

It is clearly of interest to clarify the possible benefits, if any, that can be obtained by changing $\cB$ in \eqref{eq:APclassical}, or in the corresponding definition of the Muckenhoupt $A_p$-condition for matrix weights, to some other (structured) collection of geometric objects. In this article, we will explore one specific benefit: the ability to modify invariance properties of the weights satisfying the Muckenhoupt condition. Let us be more specific; the Euclidean balls $\cB$ are clearly invariant under any affine transformation of the type $T_r:=r\cdot+c$, with $r>0$ a scalar and $c\in\bR^d$, which consequently makes the corresponding Muckenhoupt weights invariant under composition with affine maps of this particular type. 
However, the transformations $T_r$ are all isotropic in nature, and the invariance breaks down if we compose a standard Muckenhoupt weight with some other affine transformation $A\cdot+c$, where $A$ is an invertible $d\times d$-matrix that happens to be ``anisotropic''.

In the present paper, we will obtain invariance of the Muckenhoupt condition for matrix weights under certain families of anisotropic affine transformations by modifying the geometric structure of the balls in $\cB$ accordingly. 

The paper is structured as follows. In Section \ref{s:a} we define the Muckenhoupt condition for scalar and matrix weights for very general coverings $\cB$ with almost no restrictions on the geometry of the sets in $\cB$. A connection between matrix weights and scalar weights defined by the same family $\cB$ is made.

In Section \ref{sec:metr} we impose more structure on $\cB$ and require that $\cB$ is the collection of balls generated by a pseudo-metric, similar to the setting considered by Calder\'on \cite{Cald:1976a}. We show that matrix weights satisfying a Muckenhoupt condition also satisfy certain doubling conditions and a reverse H\"older inequality.

We then continue the gradual refinement of the structure of $\cB$ in Section \ref{sec:HS} and consider a subclass of pseudo-metrics that are homogeneous with respect to a one-parameter group of (possibly anisotropic) dilations. Our main contribution is to show invariance of the corresponding Muckenhoupt condition under composition of the matrix weight with affine transformations generated by the dilation group.

We present two specific results benefiting directly from such a modified invariance in Section \ref{sec:HS}. We prove a boundedness result for band-limited Fourier multipliers on the matrix-weighted space $L^p(W)$ for weights satisfying the mentioned modified Muckenhoupt condition, where the support set of the multiplier is defined by an affine transformation of a fixed compact set. 
We also prove a corresponding (nonuniform) sampling theorem for band-limited functions in this anisotropic setting.  

Finally, in Section \ref{sec:B}, we introduce the notion of matrix-weighted Besov spaces in the anisotropic setting studied in Section \ref{sec:HS}. The anisotropic matrix-weighted Besov spaces extend the construction of Besov spaces in the isotropic setting considered by Roudenko \cite{Roud:2003a} and Frazier and Roudenko \cite{FrazRoud:2004a}. A discrete characterization of vector functions in the Besov spaces is derived using the multiplier and sampling results from Section \ref{sec:HS}.  To the knowledge of the author, this construction of weighted spaces in the anisotropic setting is new even in the scalar case.  

\section{Structured bases and corresponding Muckenhoupt conditions}\label{s:a}

 In this section we consider general covering systems $\cB$ forming so-called bases for $\bR^d$ and define Muckenhoupt conditions for both scalar and matrix weights relative to such coverings.

\begin{definition}\label{d:1}
A basis for $\bR^d$, $d\in \bN$, is any collection $\cB$ of measurable, bounded subsets of $\bR^d$ with non-empty interior. A structured basis for $\bR^d$ is a basis such that there exists a sequence $E_j\in \cB$, $j\in\bN$, with $\cup_j \mathring{E}_j=\bR^d$ and $E_{j}\subseteq E_{j+1}$, $j\in \bN$.  
\end{definition}
For any structured basis $\cB$, we can now define the various Muckenhoupt conditions, where we first consider the scalar case.

Let $\cB$ be a structured basis for $\cD$, $d\in \bN$.
A (measurable) scalar weight $w:\cD\rightarrow (0,\infty)$ is said to satisfy the
 Muckenhoupt $A_p$-condition, $1<p<\infty$, provided 
 \begin{equation}\label{eq:AP}
[w]_{A_p(\cB)}:=\sup_{E\in\cB}\fint_E w(x)\,\d x\cdot\bigg[\fint_E w^{-\frac{1}{p-1}}(x)\,\d x\bigg]^{p-1}<\infty,
\end{equation}
where for any measurable subset $E\subset \bR^d$ of positive measure, 
we use the notation
$$\fint_E f(x)\,\d x:=\frac{1}{|E|}\int_E f(x)\,\d x\quad\text{and}\quad f(E):=\int_E f(x)\,\d x.$$
For $p=1$ we say that $w:\cD\rightarrow (0,\infty)$ satisfies the
 Muckenhoupt $A_1$-condition provided
\begin{equation}\label{eq:A1}
[w]_{A_1(\cB)}:=\sup_{E\in\cB}\fint_E w(x)\,\d x\cdot \esssup_{x\in E} w^{-1}(x)<\infty.
 \end{equation}
We denote the class of such Muckenhoupt weights by $A_p(\cB)$.
Even though the scalar
$A_p$-conditions are quite involved, they are still very much operational since quite large
classes of, e.g., polynomial weights can be shown to satisfy the
respective conditions for suitable $\cB_\rho$, see Lemma \ref{le:poly} of Section \ref{sec:HS}.

Next, we consider a matrix-valued function $W\colon \cD\rightarrow \bC^{N\times N}$, which is measurable
and strictly positive definite almost everywhere. We will refer to such a function as a \textit{matrix weight}. Inspired by the  definition of Muckenhoupt $A_p$-conditions for matrix weights relative to standard Euclidean balls considered in \cite{Roud:2003a,FrazRoud:2004a}, which were based on the Muckenhoupt conditions first derived in \cite{TreiVolb:1997a,Volb:1997a,NazaTrei:1996a}, we follow \cite{NielSiki:2025a} and give the following definition of the Muckenhoupt $A_p$-condition for matrix weights.

\begin{definition}
Suppose $W\colon \cD\rightarrow \C^{N\times N}$ is a matrix weight and let $0<p<\infty$.
\begin{itemize}
    \item 
In the case $0<p\leq 1$, $W$ is said to satisfy the matrix Muckenhoupt $A_p$-condition relative to $\cB$ provided 
\begin{equation}\label{eq:mA1}
	[W]_{{\mathbf{A}_p(\cB)}}:=\sup_{E\in\cB} \esssup_{t\in E} \fint_E \|W^{1/p}(x)W^{-1/p}(t)\|^p\dx<\infty.
\end{equation}
\item
In the case $1<p<\infty$, $W$ is said to satisfy the matrix Muckenhoupt $A_p$-condition relative to $\cB$ provided 
\begin{equation}\label{eq:Roudenko}
 [W]_{{\mathbf{A}_p(\cB)}}:=\sup_{E\in \cB} \bigg[\fint_E\left( \fint_E \big\|W^{1/p}(x)W^{-1/p}(t)\big\|^{p'} {\dt}\right)^{p/p'} {\dx}\bigg]^{1/p}<\infty,
\end{equation}
with $p'$ the dual exponent to $p$, i.e., $1/p+1/{p'}=1$. 
The norm $\|\cdot\|$ appearing in the integrals in Equations\ \eqref{eq:mA1} and \eqref{eq:Roudenko} is any matrix norm on the $N\times N$ matrices. We write $W\in \mathbf{A}_p(\cB)$ whenever $[W]_{{\mathbf{A}_p(\cB)}}<\infty$, $0<p<\infty$.
\end{itemize}
\end{definition}

\begin{remark}\leavevmode\label{re:dual}
\begin{itemize}
    \item It can be verified that $W\in\mathbf{A}_p(\cB)$,  $1<p<\infty$, implies $W^{-p'/p}\in \mathbf{A}_{p'}(\cB)$ with $p'$ the dual exponent to $p$, see \cite{Roud:2003a} for further details. 
    \item In the scalar case $N=1$, it is straightforward to verify that condition \eqref{eq:Roudenko} reduces to the corresponding scalar condition \eqref{eq:AP} for  $1<p<\infty$, while \eqref{eq:mA1} reduces to the scalar $A_1$ condition \eqref{eq:A1} for any $0<p\leq 1$.
\end{itemize}
\end{remark}
The matrix-weighted vector-valued $L^p$-spaces, $0<p<\infty$, will be central to the present study. For $W\colon
\cD\to\C^{N\times N}$ a matrix weight, let $L^p(W)$ denote the family of
measurable functions $\vf\colon \cD\to\C^N$ with 
\begin{equation*}
  \norm{\vf}_{L^p(W)}:=\Biggl(\int_\cD\abs{W^{1/p}(x)\vf(x)}^p\,\d x\Biggr)^{1/p}<\infty.
\end{equation*}
In order to turn $L^p(W)$ into a (quasi-)Banach space, one has to factorize over 
 $\{\vf\colon \cD\to\C^N;\norm{\vf}_{L^p(W)}=0\}$. 
For $1<p<\infty$, it can be verified that $L^p(W)$ is a Banach space and its dual space is $L^{p'}(W^{-p'/p})$, see \cite{Volb:1997a} for further details. 

For $1< p<\infty$ we also have the following important general characterization of $\mathbf{A}_p(\cB)$, where we let $\mathbf{1}_E$ denote the characteristic function of a measurable subset $E\subseteq \bR^d$.

\begin{prop}\label{prop:avg}
    Consider a matrix-valued function $W\colon \cD\rightarrow \bC^{N\times N}$, which is measurable, locally integrable, and strictly positive definite a.e. For $1<p<\infty$, we have $W\in \mathbf{A}_p(\cB)$  if and only if the averaging operators
    $$A_E\vf:=\mathbf{1}_E \fint_E \vf\, \d t,$$
    are uniformly bounded on $L^p(W)$ for $E\in\cB$. Moreover, the constant $[W]_{{\mathbf{A}_p(\cB)}}$ is equivalent to $\sup_{E\in\cB}\|A_E\|_{L^p(W)\rightarrow L^p(W)}$.
\end{prop}
The proof of Proposition \ref{prop:avg} can be found in \cite{NielSiki:2025a} (strictly speaking, \cite{NielSiki:2025a} considers only the domain $\bT^d$, but the proof extends verbatim to the setting of $\bR^d$).  Proposition \ref{prop:avg} leads to the following specific connection between Muckenhoupt matrix weights and scalar weights defined by the basis $\cB$. 
\begin{lemma}\label{le:sc}
Let $0<p<\infty$ and suppose $W\colon \cD\rightarrow \bC^{N\times N}$ is in $\mathbf{A}_p(\cB)$. Then there is a constant $C:=C(p)$ such that
\begin{itemize}
    \item[(i)] for any $x\in\bR^d$, the scalar weight $w_x(t):=|W^{1/p}(t)x|^p$ is in ${A}_p(\cB)$ and $[w_x]_{  {A}_p(\cB)}\leq C [W]_{  \mathbf{A}_p(\cB)}$ provided $1<p<\infty$,
    \item[(ii)] for any $x\in\bR^d$, the scalar weight $w_x(t):=|W^{1/p}(t)x|^p$ is in ${A}_1(\cB)$ and $[w_x]_{  {A}_1(\cB)}\leq  [W]_{  \mathbf{A}_p(\cB)}$ provided $0<p\leq 1$.
\end{itemize}
\end{lemma}
 \begin{proof} For (i), let $\vx\in\bR^d$ and take any locally integrable scalar function $\phi:\bR^d\rightarrow \bC$. Now define $\vf(t):=\phi(t)\vx$. By Proposition \ref{prop:avg} there exists $C:=C_p$ such that for any $E\in\cB$,
 \begin{align*}
     \|A_E\vf\|_{L^p(W)}^p&=\fint_E \bigg|\mathbf{1}_E\fint \phi(u)\d u\bigg|^p |W^{1/p}(t)\vx|^p\,\d t\\&\leq C^p\|\vf\|_{L^p(W)}^p\\
&=\int_{\bR^d} |\phi(t)|^p|W^{1/p}(t)\vx|^p\,\d t.
 \end{align*}
    We now reinterpret the estimate using Proposition \ref{prop:avg} in the scalar setting. We obtain that the mappings $\phi\rightarrow \mathbf{1}_E\fint \phi$ are uniformly bounded on $L^p(w_\vx)$, and we conclude from Proposition \ref{prop:avg} that $w_\vx\in {A}_p(\cB)$ with $[w_x]_{  {A}_p(\cB)}\leq C [W]_{\mathbf{A}_p(\cB)}$ for some $C$ independent of $x$.

    For (ii),  let $\vx\in\bR^d$. Take any set $E\in\cB$. We now use the estimate \eqref{eq:mA1} to obtain, 
    \begin{align*}
\fint_E |W^{1/p}(s)\vx|^p\,\d s&=\fint_E |W^{1/p}(s)W^{-1/p}(t)W^{1/p}(t)\vx|^p\,\d s\\
    &\leq\fint_E ||W^{1/p}(s)W^{-1/p}(t)||^p\,\d s\cdot |W^{1/p}(t)\vx|^p\\
    &\leq [W]_{\mathbf{A}_p(\cB)} |W^{1/p}(t)\vx|^p
    \end{align*}
  for a.e.\ $t\in E$ as desired. The estimate is uniform in $E$ proving that $|W^{1/p}(t)\vx|^p\in  {A}_1(\cB)$ with 
 $[|W^{1/p}(t)\vx|^p]_{  {A}_1(\cB)}\leq  [W]_{  \mathbf{A}_p(\cB)}$ 
\end{proof}
We will need the elementary fact that for any matrix $A\in\bC^{N\times N}$, $r>0$, and an orthonormal basis $\{\ve_j\}$ for $\bC^N$, we can estimate the matrix norm as follows
\begin{equation}\label{eq:normy}
    \|A\|^r\asymp \sum_{j=1}^N |A\ve_j|^r,
\end{equation}
where the equivalence constants depend only on $N$ and $r$.
\begin{corollary}
Let $0<p<\infty$ and suppose $W\colon \cD\rightarrow \bC^{N\times N}$ is in $\mathbf{A}_p(\cB)$. Then $\|W\|\in {A}_{\max\{1,p\}}(\cB)$.
\end{corollary}
\begin{proof}
  Let $\{\ve_j\}$ be the standard basis for $\bC^N$.  We have, using the spectral norm $\|\cdot\|$ on $\bC^{N\times N}$ and the estimate \eqref{eq:normy},
    \begin{align*}
        \|W(x)\|&=\|W^{1/p}(x)\|^{p}\\
                &\asymp \sum_{j=1}^N |W^{1/p}(x)\ve_j|^{p},
    \end{align*}
 so $\|W(x)\|$ is in ${A}_{\max\{1,p\}}(\cB)$ as it is equivalent to a sum of scalar weights in ${A}_{\max\{1,p\}}(\cB)$ by Lemma \ref{le:sc}. Here we use the elementary fact that a finite sum of ${A}_q(\cB)$ weights is again an ${A}_q(\cB)$ weight, with a constant depending only on $N$ and the individual ${A}_q(\cB)$ constants.
\end{proof}
\section{Weights relative to bases generated by pseudo-metrics}\label{sec:metr}
The structured bases considered in Section \ref{s:a} are extremely general, with almost no restrictions imposed on the sets in $\cB$. One cannot reasonably expect that every desired property of weights should hold in such generality. In this section, we make a specific choice and consider a special subfamily of structured bases for $\bR^d$ induced by a pseudo-metric in order to deduce a certain specific correspondence between scalar and matrix Muckenhoupt weights.  

Suppose we are given a distance function $\varrho:\bR^d\times\bR^d\rightarrow [0,\infty)$ satisfying
\begin{itemize}
    \item[(i)] $\varrho(x,x)=0,\qquad x\in\bR^d,$
    \item[(ii)] $\varrho(x,y)=\varrho(y,x),\qquad x,y \in\bR^d,$
    \item[(iii)] there is a constant $c$ such that $\varrho(x,z)\leq c[\varrho(x,y)+\varrho(y,z)], \qquad x,y,z \in\bR^d.$
\item[(iv)] given an open neighborhood $N$ of $x\in\bR^d$ there is an $\epsilon>0$ such that $$B_{\varrho}(x,\epsilon):=\{y\in\bR^d:\varrho(x,y)<\epsilon\}\subseteq N.$$
\item[(v)] The balls $B_{\varrho}(x,r)$ are measurable, the measure $|B_{\varrho}(x,r)|$ is continuous as a function of $r>0$, and there is a doubling constant $c'$ such that
\begin{equation}\label{eq:ddo}
|B_{\varrho}(x,2r)|\leq c'|B_{\varrho}(x,r)| 
\end{equation}
for all $x\in\bR^d$ and $r>0$.
\end{itemize}
Let us now fix such a function $\varrho$. It is clear that we can define a corresponding structured basis by letting
$$\cB_\varrho:=\big\{B_{\varrho}(x,r):x\in\bR^d,r>0\big\}.$$
Scalar Muckenhoupt weights relative to $\cB_\varrho$ have been studied in detail by Calder\'on
\cite{Cald:1976a}, see also Jawerth \cite[Example 4.1]{Jawe:1986a} and Kurtz \cite{Kurt:1975a}. A direct consequence of the $A_p(\cB_\varrho)$ condition is that any scalar Muckenhoupt weight $w\in A_p(\cB_\varrho)$, $1\leq p<\infty$, satisfies for $F\in \cB_\varrho $ and any measurable subset $E\subseteq F$ with $|E|>0$,
\begin{equation}\label{eq:ddd}
    \frac{w(F)}{w(E)}\leq [w]_{{A}_p(\cB_\varrho)}\bigg(\frac{|F|}{|E|}\bigg)^p,
\end{equation}
see \cite[Lemma 4]{Cald:1976a}. By letting $F=B_{\varrho}(x,2r)$ and $E=B_{\varrho}(x,r)$, estimate \eqref{eq:ddd}, together with 
the doubling assumption \eqref{eq:ddo},
implies a corresponding doubling condition for the measure $\d\mu=w\dt$, 
\begin{equation}\label{eq:doubb}
\int_{B_{\varrho}(x,2r)} w(t) \dt\leq C \int_{B_{\varrho}(x,r)} w(t) \dt,
\end{equation}
with $C=(c')^p[w]_{{A}_p(\cB_\varrho)}$. 
It is well-known that the doubling condition \eqref{eq:doubb} implies that for any $\lambda\geq 1$,
\begin{equation}\label{eq:dou_exp}
\int_{B_{\varrho}(x,\lambda r)} w(t) \dt\leq C_w \lambda^\beta\int_{B_{\varrho}(x,r)} w(t) \dt,
\end{equation}
with $\beta:=\log_2C$.

Estimate  \eqref{eq:doubb} has the following immediate implications for matrix weights.
\begin{lemma}\label{le:do}
   Let $0<p<\infty$ and suppose $W$ is a $\mathbf{A}_p(\cB_\varrho)$ weight. Then it holds that
   \begin{enumerate}
   \item[(i)] For $x\in\bR^d$, the weight $u_x(\cdot):=|W^{1/p}(\cdot)x|^p$ satisfies the doubling condition
   \begin{equation}\label{eq:Wdoub}
   \int_{B_{\varrho}(y,2r)} u_x(t)\,\d t\leq C\int_{B_{\varrho}(y,r)} u_x(t)\,\d t,
    \end{equation}
         with a constant $C:=C([W]_{\mathbf{A}_p(\cB_\varrho)})$ independent of $x,y\in\bR^d$ and $r>0$.
       \item[(ii)] $\|W(t)\|$ satisfies the doubling condition
       $$\int_{B_{\varrho}(x,2r)} \|W(t)\|\,\d t\leq C\int_{B_{\varrho}(x,r)} \|W(t)\|\,\d t,$$
        with a constant $C:=C([W]_{\mathbf{A}_p(\cB_\varrho)})$ independent of $x\in\bR^d$ and $r>0$.
        \item[(iii)] in case $1<p<\infty$, the weight $v_x(\cdot):=\|W^{1/p}(x)W^{-1/p}(\cdot)\|^{p'}$ satisfies the doubling condition
        $$\int_{B_{\varrho}(y,2r)} v_x(t)\,\d t\leq C\int_{B_{\varrho}(y,r)} v_x(t)\,\d t,$$
         with a constant $C:=C([W]_{\mathbf{A}_p(\cB_\varrho)})$ independent of $x,y\in\bR^d$ and $r>0$.
   \end{enumerate}
\end{lemma}
\begin{proof}
Claim (i) is a consequence of \eqref{eq:doubb} and the fact that $|W^{1/p}(t)x|^p\in A_{\max\{1,p\}}(\cB_\varrho)$ with $A_{\max\{1,p\}}(\cB_\varrho)$ constant that depends only on $[W]_{\mathbf{A}_p(\cB_\varrho)}$. 
For (ii) we fix an orthonormal basis $\{\ve_j\}$ for $\bC^N$. By \eqref{eq:normy},
\begin{align*}
\int_{B_{\varrho}(x,2r)} \|W(t)\|\,\d t&\asymp  \sum_{j=1}^N\int_{B_{\varrho}(x,2r)} |W^{1/p}(t)\ve_j|^p\,\d t\\
&\leq C\sum_{j=1}^N\int_{B_{\varrho}(x,r)} |W^{1/p}(t)\ve_j|^p\,\d t\\
&\asymp\int_{B_{\varrho}(x,r)} \|W(t)\|\,\d t,
\end{align*}
where we have again used the estimate \eqref{eq:doubb} together with the fact that $|W^{1/p}(t)\ve_j|^p\in A_{\max\{1,p\}}(\cB_\varrho)$ with an $A_{\max\{1,p\}}(\cB_\varrho)$ bound that depends only on $[W]_{\mathbf{A}_p(\cB_\varrho)}$.
Finally, for (iii), we use the fact that for two self-adjoint operators $P,Q$ on a Hilbert space, $\|PQ\|=\|QP\|$, which implies that $$v_x(\cdot):=\|W^{1/p}(x)W^{-1/p}(t)\|^{p'}=
\|W^{-1/p}(t)W^{1/p}(x)\|^{p'}.$$
Hence, by the same argument as in (ii),
\begin{align*}
\int_{B_{\varrho}(y,2r)} v_x(t)\,\d t&\asymp  \sum_{j=1}^N\int_{B_{\varrho}(y,2r)} |W^{-1/p}(t)W^{1/p}(x)\ve_j|^{p'}\,\d t\\
&\leq C\sum_{j=1}^N\int_{B_{\varrho}(y,r)} |W^{-1/p}(t)W^{1/p}(x)\ve_j|^{p'}\,\d t\\
&\asymp\int_{B_{\varrho}(y,r)} v_x(t)\,\d t,
\end{align*}
where we used that $W^{-p'/p}\in \mathbf{A}_{p'}(\cB_\varrho)$ since $W\in \mathbf{A}_p(\cB_\varrho)$, cf.\ Remark \ref{re:dual}.
\end{proof}

One of the main results derived by Calder\'on for weights relative to $\cB_\varrho$, see \cite[Theorem 1]{Cald:1976a}, is the following reverse H\"older estimate.
\begin{prop}\label{prop:revH}
  Let  $1\leq p<\infty$ and  suppose $w\in A_p(\cB_\varrho)$. Then there exist constants $r:=r([w]_{A_p(\cB_\varrho)},p,\varrho)>1$ and $c_1:=c_1([w]_{A_p(\cB_\varrho)},p,\varrho)>0$ such that
    \begin{equation}
        \bigg[\fint_B w^r\dx\bigg]^{1/r}\leq c_1\fint_B w\dx,\qquad B\in \cB_\varrho.
    \end{equation}
\end{prop}
This leads to the following new reverse H\"older estimate for matrix weights.  Given $B\in \cB_\varrho$ and $0<p<\infty$, we may consider the (quasi-)norm 
$$\eta_{p,B}(x)=\bigg(\fint_B |W^{1/p}(t)x|^p\dt\bigg)^{1/p}$$
on $\bC^N$. It is a consequence of a theorem by F.\ John, see \cite{John:2014a}, that there exists a corresponding sequence of positive definite $N\times N$ matrices $\{A_B\}_{B\in\cB_\varrho}$, called a corresponding sequence of reducing operators, such that
$$\eta_{p,B}(x)\asymp |A_Bx|$$
uniformly for $B\in \cB_\varrho$ and $x\in\bC^N$.  We refer to \cite[Section 5]{FrazRoud:2004a} for a detailed discussion of this result, including a discussion on the non-trivial adaptation to the quasi-normed case when $0<p<1$. In a similar fashion, we may associate reducing operators 
$\{A_B^{\#}\}_{B\in\cB_\varrho}$ to the (quasi-)norms
$$\eta_{p,B}^\prime(x)=\bigg(\fint_B |W^{-1/p}(t)x|^{p'}\dt\bigg)^{1/p'}.$$
Interestingly, for $1<p<\infty$, the matrix $A_p$-condition can be expressed in terms of the reducing operators as having the following uniform bound,
\begin{equation}\label{eq:reduc}
\sup_{B\in\cB_\varrho} \|A_BA_B^{\#}\|\leq C_p<\infty.
\end{equation}
The proof of the equivalence of \eqref{eq:reduc} with the $A_p$-condition is given in \cite[Section 3]{Roud:2003a} in  the case where $\cB$ is the family of  Euclidean balls. The reader can easily verify that the proof extends verbatim to any structured basis $\cB$.

We have the following reverse H\"older inequalities for $\mathbf{A}_p(\cB_\varrho)$-weights.

\begin{prop}
    Let $W$ be a $\mathbf{A}_p(\cB_\varrho)$ weight for some $1<p<\infty$. Then there exist $\delta>0$ and constants $C_q$ such that for all $B\in\cB_\varrho$,
    \begin{align}
        \fint_B \|W^{1/p}(t)A_B^{\#}\|^q\dt &\leq C_q,\qquad\text{for all } q<p+\delta\label{eq:PA1}\\
         \fint_B \|A_BW^{-1/p}(x)\|^q\dx &\leq C_q,\qquad\text{for all } q<p'+\delta.\label{eq:PA2}
    \end{align}
\end{prop}

\begin{proof}
We first prove \eqref{eq:PA2}. Since $W\in\mathbf{A}_p(\cB_\varrho)$, we have
$W^{-p'/p}\in\mathbf{A}_{p'}(\cB_\varrho)$, cf.\ Remark \ref{re:dual}. For the dual weight
$W^{-p'/p}$, the operators $A_B$ play the role of the reducing operators associated with
the $p'$-averages of $W^{-1/p}$, up to  dimensional constants.

Let $\{\ve_j\}$ be the standard basis for $\bC^N$. Applying Lemma \ref{le:sc} to the
matrix weight $W^{-p'/p}$ gives that
\[
        v_j(x):=|W^{-1/p}(x)A_B\ve_j|^{p'}
\]
is a scalar $A_{p'}(\cB_\varrho)$-weight, with Muckenhoupt constant independent of
$j$ and $B$. Thus, for $q>p'$ with $q/p'\leq r'$, where $r'>1$ is the reverse H\"older
exponent from Proposition \ref{prop:revH} applied to these scalar $A_{p'}$-weights, we obtain
\begin{align*}
     \fint_B \|A_BW^{-1/p}(x)\|^q\dx
     &\asymp \fint_B \|W^{-1/p}(x)A_B\|^q\dx\\
     &\asymp  \sum_{j=1}^N\fint_B |W^{-1/p}(x)A_B\ve_j|^q\dx\\
     &\leq c \sum_{j=1}^N\left(\fint_B |W^{-1/p}(x)A_B\ve_j|^{p'}\dx\right)^{q/p'}\\
     &\asymp \left(\fint_B \|W^{-1/p}(x)A_B\|^{p'}\dx\right)^{q/p'}\\
     &\asymp \|A_B^{\#}A_B\|^q
     \leq cC_p^q.
\end{align*}
Here we used that $A_B$, $A_B^{\#}$, and $W^{-1/p}(x)$ are positive definite, so
$\|A_BW^{-1/p}(x)\|=\|W^{-1/p}(x)A_B\|$, and then the reducing-operator
characterization \eqref{eq:reduc}.

The proof of \eqref{eq:PA1} is identical but applied directly to $W$: the scalar weights
\[
        w_j(t):=|W^{1/p}(t)A_B^{\#}\ve_j|^p
\]
belong to $A_p(\cB_\varrho)$ uniformly in $B$ and $j$, and Proposition \ref{prop:revH}
therefore gives \eqref{eq:PA1} for all $q<p+\delta_1$ for some $\delta_1>0$. Combining
this gain with the dual gain above and taking
\[
        \delta:=\min\{\delta_1,\,p'(r'-1)\}
\]
completes the proof.
\end{proof}

In the isotropic case, reverse H\"older estimates form the foundation for obtaining maximal function estimates in the matrix-weighted setting, see Christ and Goldberg \cite{ChriGold:2001a} and Goldberg \cite{Gold:2003a}. However, it will take us too far astray here to pursue the study of maximal functions in the anisotropic setting.
\section{A homogeneous structure on $\bR^d$}\label{sec:HS}
We now refine the setting from Section \ref{sec:metr} even further by considering homogeneous type spaces on $\bR^d$  created using a quasi-norm induced by a
one-parameter group of dilations. The quasi-norms turn out to form a sub-family of the pseudo-metrics considered in Section~\ref{sec:metr}.   

As before, we let $|\cdot|$ denote the Euclidean norm on $\bR^d$ induced by the inner product
$\langle \cdot,\cdot\rangle$. We assume that $A$ is a real symmetric $d\times d$ matrix with strictly positive eigenvalues. For $t>0$ define the group of dilations $\delta_t:\bR^d\to\bR^d$
 by $\delta_t := \exp (A\ln t)$ and let $\nu:=\textrm{trace}(A)$. Notice that $\delta_t$ is real and symmetric for any $t>0$. The matrix $A$ will be kept
fixed throughout the paper. Some well-known properties of $\delta_t$ are (see
\cite[Part II]{SteiWain:1978a}),
\begin{itemize}
\item $\delta_{ts}=\delta_t\delta_s$.
\item $\delta_1 = Id$ (identity on $\bR^d$).
\item $\delta_t \xi$ is jointly continuous in t and $\xi$, and $\delta_t\xi \to 0$ as $t\to
0^+$.
\item $|\delta_t|:=\textrm{det}(\delta_t)=t^\nu.$
\end{itemize}
According to \cite[Proposition 1.7]{SteiWain:1978a} there exists a strictly positive
symmetric matrix $P$ such that for all $\xi \in \bR^d$,
\begin{equation}\label{eq:P}
[\delta_t\xi]_P:=\langle P\delta_t\xi,\delta_t\xi \rangle^{\tfrac{1}{2}}
\end{equation}
is a strictly increasing function of $t$. This helps us introduce a quasi-norm $|\cdot|_A$ associated with $A$.
\begin{definition}\label{def:A}
We define the function $|\cdot|_A:\bR^d\to \bR_+$ by $|0|_A:=0$ and for
$\xi\in\bR^d\backslash\{0\}$ by letting $|\xi|_A$ be the unique solution $t$ to the
equation $[\delta_{1/ t}\xi]_P=1$.
\end{definition}

For notational convenience, let $\brac{\cdot}_A:=1+|\cdot|_A$ denote the \textit{bracket function}. It can be verified, see, e.g., \cite{SteiWain:1978a, BoruNiel:2008a}, that:
\begin{itemize}
\item There exists a constant $C_A\geq 1$ such that
\begin{equation}\label{windingroads}
|\xi+\zeta|_A\le C_A(|\xi|_A+|\zeta|_A), \, \xi,\zeta \in \bR^d.
\end{equation}
\item We have the homogeneity property: $|\delta_t\xi|_A=t|\xi|_A$, $t>0$.
\item There exists constants $C_1,C_2,\alpha_1,\alpha_2>0$ such that
\begin{equation}\label{natasha}
C_1\min(|\xi|^{\alpha_1}_A,|\xi|^{\alpha_2}_A) \le |\xi|\le
C_2\max(|\xi|^{\alpha_1}_A,|\xi|^{\alpha_2}_A), \qquad \xi \in \bR^d.
\end{equation}
\item For any $\epsilon>0$, there exists $C_\epsilon<\infty$ such that
\begin{equation}\label{eq:integral}
\int_{\bR^d}\brac{x}_A^{-\nu-\epsilon}\,\d x\leq C_\epsilon\quad \text{and}\quad \sum_{k\in\bZ^d}\brac{t-k}_A^{-\nu-\epsilon}\leq C_\epsilon,
\end{equation}
for $t\in\bR^d$.
\end{itemize}
In fact, let $\sigma(A)$ denote the spectrum of $A$. Then we may select
$\alpha_1$ and $\alpha_2$ in \eqref{natasha} as follows,
\begin{equation}
  \label{eq:spec}
  \alpha_1=\min_{\lambda\in\sigma(A)} \lambda\leq \alpha_2=\max_{\lambda\in\sigma(A)} \lambda.
\end{equation}

\begin{example}
For $A=\textrm{diag}(\beta_1,\beta_2,\ldots,\beta_d)$, $\beta_i>0$, we have
$\delta_t=\textrm{diag}(t^{\beta_1},t^{\beta_2},\ldots,t^{\beta_d})$, and one may verify
that
\begin{equation*}
|\xi|_A\asymp \sum_{j=1}^d|\xi_j|^{\tfrac{1}{\beta_j}}, \qquad \xi \in\bR^d.
\end{equation*}
In particular, for $A=\textrm{diag}(1,1,\ldots,1)$, we recover (up to equivalence) the usual isotropic Euclidean norm on $\bR^d$.
\end{example}
\begin{remark}
The same construction works for $A$ a real matrix with eigenvalues having strictly positive real parts, but this added generality comes at the expense of having to also consider the quasi-distance induced by $B=A^\top$. We  shall not pursue this level of generality here.
\end{remark}

We define the balls ${B}_A(\xi,r):=\{\zeta\in\bR^d:|\xi-\zeta|_A<r\}$. It
can be verified that $|{B}_A(\xi,r)|=r^{\nu}\omega_d^A$, where
$\omega_d^A:=|{B}_A(0,1)|$, and consequently $(\bR^d,|\cdot|_A,d\xi)$ is a space of homogeneous
type with homogeneous dimension $\nu$.

We let 
\begin{equation}\label{eq:ball}\mathcal{B}_A:=\{{B}_A(\xi,r):\xi\in\bR^d,r>0\} \end{equation}
denote the collection of all such balls. We notice that $\mathcal{B}_A$ is invariant under affine transformations of the type $T_\eta:=\delta_\eta\cdot+c$ with $\eta>0$ and $c\in\bR^d$. In fact, 
one easily verifies that
\begin{equation}\label{eq:aff}
T_\eta {B}_A(\xi,r)={B}_A(\delta_\eta\xi+c,r\eta).    
\end{equation}
In particular, for $r>0$ and $\xi\in \bR^d$, we have
\[{B}_A(\xi,r)=T_r {B}_A(0,1)\]
with $T_r=\delta_r\cdot+\xi$. 

Moreover, it is clear that $$\varrho_A(\xi,\eta):=|\xi-\eta|_A,\qquad \xi,\eta\in\bR^d,$$ induces a quasi-metric that satisfies conditions (i)-(v) in Section \ref{sec:metr}, so all results derived in the previous sections apply to the structured basis $\cB_A$.  Let us also note that the doubling condition can be stated precisely in this case, for $\lambda\geq 1$,
\begin{equation}\label{eq:precise_doubling}
\frac{|B_A(x,\lambda r)|}{|B_A(x,r)|}=\frac{(\lambda r)^{\nu}\omega_d^A}{r^{\nu}\omega_d^A}=\lambda^\nu.
\end{equation}
Equation~\eqref{eq:precise_doubling} implies the following more precise version of \eqref{eq:dou_exp}; for $\lambda\geq 1$,
\begin{equation}\label{eq:doubb2}
\int_{B_{A}(x,\lambda r)} w(t) \dt\leq [w]_{{A}_p(\cB_A)}\lambda^{\nu p} \int_{B_{A}(x,r)} w(t) \dt
\end{equation}
for $w\in A_p(\cB_A)$, $1\leq p <\infty$.

 \begin{remark}
  By multiplying the matrix $P$ in \eqref{eq:P} by a suitable positive scalar, it is always possible to scale $|\cdot|_A$ such that
  \begin{equation}\label{eq:scaling}
  B_A(0,1)\subseteq \{z\in\bR^d:|z|<2\}.
  \end{equation}
  We shall assume that \eqref{eq:scaling} holds in the sequel.
 \end{remark}

It turns out that the Muckenhoupt $A_p$-classes for the family $\cB_A$ contain an abundance of non-trivial polynomial weights, both in the scalar and matrix case. We have the following lemma, which turns out to be a direct consequence of an important uniform norm estimate for multivariate polynomials  obtained by  Ricci and Stein \cite{RiccStei:1987a}.

\begin{lemma}\label{le:poly}
  Let $$P(x)=\sum_{\alpha\in(\bN\cup \{0\})^d:|\alpha|\leq k} c_\alpha x^\alpha,\qquad x\in\bR^d,$$
  be a polynomial of degree (at most) $k\in\bN$ and let $1<p<\infty$. Then, for $\beta>0$, $|P|^\beta$ belongs to $A_p(\cB_A)$ provided $k\beta<p-1$ with an $A_p(\cB_A)$-bound that depends only on $k$ and $\beta$.
  
\end{lemma}
\begin{proof}
Let $B_A(c,r)\in \cB_A$ and recall that $B_A(c,r)=T_r B_A(0,1)$, with $T_r\cdot:=\delta_r\cdot+c$. Then, by the change of variable $x=T_ru$,
\begin{align*}
    \fint_{B_A(c,r)} |P(x)|^\beta\dx \bigg[\fint_{B_A(c,r)} &|P(x)|^{-\frac{\beta}{p-1}}\dx\bigg]^{p-1} \\
    &=
    \fint_{B_A(0,1)} |P(T_ru)|^\beta\du \bigg[\fint_{ B_A(0,1)} |P(T_ru)|^{-\frac{\beta}{p-1}}\du\bigg]^{p-1}.
\end{align*}
Recall that the polynomials of degree $k$ are invariant under composition with affine transformations, so $P\circ T_r $ is also a polynomial of degree $k$. Now, as 
$B_A(0,1)$ is clearly a bounded set, we can find $\eta_0>0$ such that $B_A(0,1)\subseteq B_{\rho_0}:=\{\xi:|\xi|<\eta_0\}$. Hence,
 \begin{align*} 
   \fint_{B_A(0,1)} |P(T_ru)|^\beta\du \bigg[\fint_{ B_A(0,1)} &|P(T_ru)|^{-\frac{\beta}{p-1}}\du\bigg]^{p-1}\\
   &\leq 
    c\fint_{B_{\rho_0}} |P(T_ru)|^\beta\du \bigg[\fint_{ B_{\rho_0}} |P(T_ru)|^{-\frac{\beta}{p-1}}\du\bigg]^{p-1}\\
    &\leq cC_p,
\end{align*}
with $C_p$ the $A_p$ Muckenhoupt constant for $|P(T_ru)|^\beta$ in the standard isotropic case. It is known that $C_p$ depends only on $k$ and $\beta$, see \cite[\S 6.5]{Stei:1993a}.
\end{proof}

The polynomials of degree $k$ are invariant under composition with an invertible affine transformation, so Lemma \ref{le:poly} provides a large class of scalar $A_p(\cB_A)$ weights that are invariant under such transformations. Also, for this type of polynomial weights, the composition with an affine transformation does not modify the Muckenhoupt constant. 

 We infer from Lemma \ref{le:poly} that $\mathbf{A}_p(\cB_A)$  contains matrix weights with polynomial entries. Two easy examples  are provided by diagonal matrices with polynomial weights of the type considered in Lemma \ref{le:poly} and non-diagonal matrices with such polynomial entries satisfying a suitable diagonal-dominance condition (see \cite{NielRasm:2018a} for a discussion of diagonal-dominance conditions  for weights in an isotropic setting). 

The following lemma shows that we have a related type of affine invariance valid for any matrix weight in $\mathbf{A}_p(\cB_A)$. The specific invariance  obtained in Lemma \ref{le:inv} will be fundamental in the sequel.
\begin{lemma}\label{le:inv}
Suppose $0<p<\infty$. Let $W\in \mathbf{A}_p(\cB_A)$ and let $T_r:=\delta_r\cdot+c$ for some $r>0$ and $c\in\bR^d$. Then we have,
    $$[W\circ T_r]_{\mathbf{A}_p(\cB_A)}= [W]_{\mathbf{A}_p(\cB_A)}.$$
\end{lemma}

\begin{proof}
 Take any $W\in \mathbf{A}_p(\cB_A)$, where we first suppose $1<p<\infty$. Form the matrix weight $V:=W\circ T_r$. One verifies directly that $V^{\pm 1/p}=W^{\pm 1/p}\circ T_r$. Then for $B\in\cB_A$, by a simple change of variables,
 \begin{align*}
\fint_B\bigg( \fint_B \big\|V^{1/p}(x)&V^{-1/p}(t)\big\|^{p'} {\dt}\bigg)^{p/p'} {\dx}\\
&=
\int_{B}\left( \int_{B} \big\|W^{1/p}( T_rx)W^{-1/p}( T_rt)\big\|^{p'} \frac{\d t}{|B|}\right)^{p/p'} \frac{\d x}{|B|}\\
&=\int_{{T_rB}}\left( \int_{{T_rB}} \big\|W^{1/p}(y)W^{-1/p}(s)\big\|^{p'} \frac{\d s}{|T_rB|}\right)^{p/p'} \frac{\d y}{|T_rB|}\\
&\leq [W]_{\mathbf{A}_p(\cB_A)},
\end{align*}
where the last estimate  relies on \eqref{eq:Roudenko} and the fact that $T_rB\in \cB_A$, which follows from the geometric invariance property stated in Equation~\eqref{eq:aff}.
 This estimate together with \eqref{eq:Roudenko} implies that $$[W\circ T_r]_{\mathbf{A}_p(\cB_A)}\leq [W]_{\mathbf{A}_p(\cB_A)}.$$
Since  $T_r^{-1}=\delta_{1/r}\cdot-\delta_{1/r}c$ is an affine map with similar structure, repeating the argument yields $[W]_{\mathbf{A}_p(\cB_A)}\leq [W\circ T_r]_{\mathbf{A}_p(\cB_A)}$ proving the result for $1<p<\infty$. The argument in the case $0<p\leq 1$ is similar, but based on the condition given in Equation~\eqref{eq:mA1}. We leave the details to the reader.
\end{proof}
\subsection{Fourier multipliers and sampling}
Let us now consider two scenarios where the invariance obtained in Lemma \ref{le:inv} can be put to specific use. 
We first consider Fourier multipliers in the matrix-weighted setting. 

Let us first specify our chosen normalization of the Fourier transform. For  $f\in L^1(\bR^d)$, we let 
\begin{equation}\label{def:f}\mathcal{F}(f)(\xi):=(2\pi)^{-d/2}\int_{\bR^d}
f(x)e^{- i x\cdot\xi}\dx,\qquad \xi\in\bR^d,
\end{equation}
denote the Fourier transform and we use the standard notation $\hat{f}(\xi)=\mathcal{F}(f)(\xi)$. With this normalization, the Fourier transform extends to a unitary transform on $L^2(\bR^d)$ and we denote the inverse Fourier transform by $\mathcal{F}^{-1}$. We let $\dS$ denote the usual Fr\'{e}chet  space of Schwartz functions on $\bR^d$, while $\ddS$ denotes the corresponding dual space of tempered distributions.

For $B=B_A(c,R)\in \cB_A$, we define the following class of vector-valued functions with band-limited coordinate functions, 
\begin{equation}
E_B:=\{\vf:\bR^d\rightarrow \bC^N: f_i\in \ddS \text{ and } \text{supp}(\hat{f}_i)\subseteq B, i=1,\ldots,N\}.
\end{equation}

A Fourier multiplier is a function $m\in L^\infty(\bR^d)$ that induces a corresponding bounded multiplier operator on $L^2(\bR^d)$,
$$m(D)f:=\mathcal{F}^{-1}(m\mathcal{F}f),\qquad f\in L^2(\bR^d).$$

We will need some additional notation for the proof of the multiplier result.
Pick $r_0>0$ such that the cube $\big[-\frac{1}{2},\frac{1}{2}\big)^d$ is contained in $B_A(0,r_0)$, which is possible by the estimate in \eqref{natasha}. Then define
 \begin{equation}\label{eq:Qk}
U_k:=B_A(0,r_0)+ k, \qquad k\in\bZ^d.
\end{equation}

Now, clearly, $U_k=U_\ell+k-\ell$, $k,\ell\in\bZ^d$, and whenever $U_k\cap U_\ell\not=\emptyset$, then $k-\ell\in B_A(0,r_0)+B_A(0,r_0)$,  and consequently  $|k-\ell|_A\leq 2C_Ar_0$. It follows that there exists $n_0:=n_0(r_0)<\infty$ such that
\begin{equation}\label{eq:height}1\leq \sum_{k\in\bZ^d} \mathbf{1}_{U_k}(z)\leq n_0,\qquad z\in\bR^d.
\end{equation}

We have the following Fourier multiplier result in the matrix-weighted setting.

\begin{proposition}\label{prop:main}
Let $0<p<\infty$ and suppose $W\in\mathbf{A}_p(\cB_A)$. Let $B=B_A(c,R)\in\cB_A$ and put $\beta:=\max\{\nu,\nu p\}$.  Suppose there is a constant $K$ such that the compactly supported function $\phi:B_A(c,R)\rightarrow\bC$ satisfies 
\begin{equation}\label{eq:e1}
|\mathcal{F}^{-1}(\phi)(x)|\leq K R^{\nu} (1+R|x|_A)^{-M},\qquad x\in\bR^d, 
\end{equation}
for some $$M>\max\big\{\nu+\max\{0,p-1\}\beta,(\nu+\beta)/\min\{1,p\}\big\}.$$  Then there exists a finite constant $C:=C([W]_{\mathbf{A}_p(\cB_A)},K,p)$ such that the Fourier multiplier
$$\phi(D)f:=\mathcal{F}^{-1}[\phi\cdot \mathcal{F}(f)],$$ defined for $f\in \ddS $ with $\text{supp}(\hat{f})\subseteq B_A(c,R)$, satisfies
$$\|\phi(D)\vf\|_{L^p(W)}\leq C \|\vf\|_{L^p(W)}$$
for all $\vf\in E_B\cap L^p(W)$.
\end{proposition}
\begin{proof}
Let us assume that the result holds in the special case $B_0=B_A(0,1)$. For  general $B=B_A(c,R)$, we can then consider the multiplier $\psi:B_A(0,1)\rightarrow\bC$ defined by
$\psi(\cdot)=\phi(\delta_{R}\cdot+c)$. By a direct calculation, we obtain
$$\mathcal{F}^{-1}(\psi)(x)={R}^{-\nu} e^{-ix\cdot \delta_{R}^{-1}c}\mathcal{F}^{-1}(\phi)(\delta_{R}^{-1}x).$$
Hence, 
\begin{align*}
|\mathcal{F}^{-1}(\psi)(x)|&\leq {R}^{-\nu} R^\nu K (1+R|\delta_{R}^{-1}x|_A)^{-M}\\
&=K \left(1+\frac{R}{R}|x|_A\right)^{-M}\\
&= K  (1+|x|_A)^{-M}.
\end{align*}
Now, notice that for any $\vf\in E_B$, we have 
\begin{equation}\label{eq:vg}
\vg:={R}^{-\nu} e^{-ix\cdot \delta_{R}^{-1}c}\vf(\delta_{R}^{-1}\cdot)\in E_{B_0},
\end{equation}
since
$$\mathcal{F}(\vg)(\cdot)=\hat{\vf}\big(\delta_{R}\cdot+c\big).$$
Also notice that the scalar factor $e^{-ix\cdot \delta_{R}^{-1}c}$  in \eqref{eq:vg} is unimodular, so we have $|\vg(\cdot)|=R^{-\nu}|\vf(\delta_{R}^{-1}\cdot)|$. Now, using the assumption that the result holds for the support set $B_0$, we obtain for $\vf\in E_R$,
\begin{align*}
\|\phi(D)\vf\|_{L^p(W)}^p&=\int_{\bR^d} |W^{1/p}(t)\phi(D)\vf(t)|^p\,dt\\
&={R}^{-\nu}\int_{\bR^d} |W^{1/p}(\delta_{R}^{-1}u)(\phi(D)\vf)(\delta_{R}^{-1}u)|^p\,du\\
&={R}^{\nu p-\nu}\int_{\bR^d} |W^{1/p}(\delta_{R}^{-1}u)(\psi(D)\vg)(u)|^p\,du\\
&\leq C {R}^{\nu p-\nu}\int_{\bR^d} |W^{1/p}(\delta_{R}^{-1}u)\vg(u)|^p\,du\\
&=C {R}^{-\nu}\int_{\bR^d} |W^{1/p}(\delta_{R}^{-1}u)\vf(\delta_{R}^{-1}u)|^p\,du\\
&=C \int_{\bR^d} |W^{1/p}(t)\vf(t)|^p\,dt,
\end{align*}
where $C:=C([W(\delta_{R}^{-1}\cdot)]_{\mathbf{A}_p(\cB_A)},\kappa,p)$ and we have repeatedly used that the scalar factor $e^{-ix\cdot \delta_{R}^{-1}c}$ appearing in \eqref{eq:vg} is unimodular. However, as Lemma \ref{le:inv} shows,  the matrix ${A}_p$-condition is invariant under affine transformations in the sense that $[W(\delta_{R}^{-1}\cdot)]_{\mathbf{A}_p(\cB_A)}=[W]_{\mathbf{A}_p(\cB_A)}$, making $C$ independent of $R$ and $c$. Hence, all that remains is to prove the result in the special case when $B=B_0$. 

Now let $B=B_0$. By assumption,
$$|\mathcal{F}^{-1}(\phi)(x)|\leq K \brac{x}_A^{-M},\qquad x\in\bR^d.$$
  For any $\vf\in E_{B_0}\cap  L^p(W)$, we notice that $\phi(D)\vf=(\mathcal{F}^{-1}\phi) * \vf$, and for $t,u\in \bR^d$, 
 using the inclusion \eqref{eq:scaling}, we may use the sampling representation (see, e.g.,  \cite[Lemma 6.10]{FrazJaweWeis:1991a}), $$(\mathcal{F}^{-1}\phi * \vf)(t)=\sum_{\ell\in \bZ^d} \vf(\ell+u)[\mathcal{F}^{-1}\phi](t-u-\ell),$$
Hence,
$$W^{1/p}(\delta_{R}^{-1}t)(\mathcal{F}^{-1}\phi * \vf)(t)=\sum_{\ell\in \bZ^d} W^{1/p}(\delta_{R}^{-1}t)\vf(\ell+u)[\mathcal{F}^{-1}\phi](t-u-\ell),$$
which implies that,
\begin{align*}
|W^{1/p}(\delta_{R}^{-1}t)(\mathcal{F}^{-1}\phi * \vf)(t)|&\leq \sum_{\ell\in \bZ^d} |W^{1/p}(\delta_{R}^{-1}t)\vf(\ell+u)[\mathcal{F}^{-1}\phi](t-u-\ell)|\\
&\leq c_M\sum_{\ell\in \bZ^d} |W^{1/p}(\delta_{R}^{-1}t)\vf(\ell+u)|\brac{t-u-\ell}_A^{-M}.
\end{align*}
The proof can now be completed by using Lemma \ref{le:help2} below with $V=W(\delta_{R}^{-1}\cdot)$.
\end{proof}
\begin{lemma}\label{le:help2}
Let $V\in\mathbf{A}_p(\cB_A)$ for some $0<p<\infty$ and put $\beta:=\max\{\nu,\nu p\}$. Suppose the estimate
\begin{align}\label{eq:start}
|V^{1/p}(t)(\mathcal{F}^{-1}\phi * \vf)(t)|\leq c_M\sum_{\ell\in \bZ^d} |V^{1/p}(t)\vf(\ell+u)|\brac{t-u-\ell}_A^{-M},\qquad t,u\in\bR^d,
\end{align}
holds uniformly for $\vf\in E_{B_A(0,1)}$  for some  $$M>\max\big\{\nu+\max\{0,p-1\}\beta,(\nu+\beta)/\min\{1,p\}\big\}.$$ Then
there exists $C:=C([V]_{\mathbf{A}_p(\cB_A)})$ such that
$$\|\mathcal{F}^{-1}\phi * \vf\|_{L^p(V)}\leq C\| \vf\|_{L^p(V)},$$
for $\vf\in E_{B_A(0,1)}\cap L^p(V)$.
\end{lemma}

\begin{proof}
Let us first consider the case $0<p\leq 1$. Equation~\eqref{eq:start} yields

\begin{align}\label{eq:start2}
|V^{1/p}(t)(\mathcal{F}^{-1}\phi * \vf)(t)|^p\leq c_M^p\sum_{\ell\in \bZ^d} |V^{1/p}(t)\vf(\ell+u)|^p\brac{t-u-\ell}_A^{-Mp},
\end{align}
Let  $U_k$ be given in \eqref{eq:Qk}. We now average the estimate \eqref{eq:start2} over $u\in U_0$, where we notice, using \eqref{windingroads},
$$1+|t-\ell|_A\leq 1+C_A(r_0+|t-u-\ell|_A)\leq C_A'(1+|t-u-\ell|_A).
$$
Hence, with $c:=c_M^p(C_A')^{Mp}/|U_0|$,
\begin{align}
|V^{1/p}(t)(\mathcal{F}^{-1}\phi * \vf)(t)|^p&\leq c\int_{U_0}\sum_{\ell\in \bZ^d} 
|V^{1/p}(t)\vf(\ell+u)|^p\brac{t-\ell}_A^{-Mp}\du\nonumber\\
&= c\sum_{\ell\in \bZ^d} \brac{t-\ell}_A^{-Mp} \int_{U_0}
|V^{1/p}(t)\vf(\ell+u)|^p\du\nonumber\\
&= c\sum_{\ell\in \bZ^d} \brac{t-\ell}_A^{-Mp} \int_{U_\ell}
|V^{1/p}(t)\vf(y)|^p\dy,\label{eq:esti2}
\end{align}
where we have used Tonelli's theorem.   By the doubling condition \eqref{eq:doubb2} satisfied by the scalar weight $v_{\vw}(t):=|V^{1/p}(t)\vw|^p\in A_{\max\{1,p\}}(\cB_A)$, with constant independent of $\vw\in\bC^N$, there exists a finite constant $c_w$ such that for any $k,\ell\in\bZ^d$, and $y\in\bR^d$,
\begin{equation}\label{eq:doub}
\int_{U_k} |V^{1/p}(t)\vf(y)|^p\,dt\leq c_w\brac{k-\ell}_A^\beta\int_{U_\ell} |V^{1/p}(t)\vf(y)|^p\,dt,
\end{equation} 
with $\beta=\max\{\nu,\nu p\}$,
which follows from the observation that $$U_k\subseteq B_A\big(\ell,C_A(r_0+|k-\ell|_A)\big)\subseteq B_A\big(\ell,b(1+|k-\ell|_A)\big),$$
with $b=\max\{C_A,C_Ar_0\}$.

For $k\in\bZ^d$, we integrate inequality \eqref{eq:esti2} over $t\in U_k$, using Tonelli's theorem once more together with the estimate \eqref{eq:doub}, and the observation that there exists $c>0$ such that for all $t\in U_k$ and $\ell\in \bZ^d$,
$1+|k-\ell|_A\leq c(1+|t-\ell|_A)$,
\begin{align*}
\int_{U_k} |V^{1/p}(t)(\phi(D)\vf)(t)|^p\dt&\leq c_wc
 \int_{U_k}\sum_{\ell\in \bZ^d} \brac{k-\ell}_A^{-Mp} \int_{U_\ell}
|V^{1/p}(t)\vf(y)|^p\dy\dt\\
&\leq c'\sum_{\ell\in \bZ^d} \brac{k-\ell}_A^{-Mp+\beta} \int_{U_\ell}\int_{U_\ell}
|V^{1/p}(t)\vf(y)|^p\dt\dy,
\end{align*}
with $c':=c_wc$. In the following estimate, the matrix $A_p$-condition \eqref{eq:mA1} for $V$ will be essential. We first notice that by \eqref{eq:height}, $\{U_k\}_k$ covers $\bR^d$ and using the assumption that $Mp-\beta>\nu$ we may use \eqref{eq:integral} to define a finite constant $L:=\sum_{k\in\bZ^d}\brac{k}_A^{-Mp+\beta}<\infty$. We then have,

\begin{align*}
\|\phi(D)\vf\|_{L^p(V)}^p&\leq \sum_{k\in \bZ^d} \int_{U_k} |V^{1/p}(t)(\phi(D)\vf)(t)|^p\dt\\
&\leq c'\sum_{k\in \bZ^d}\sum_{\ell\in \bZ^d} \brac{k-\ell}_A^{-Mp+\beta}\int_{U_\ell}\int_{U_\ell}
|V^{1/p}(t)\vf(y)|^p\dt\dy\\
&=Lc'\sum_{\ell\in \bZ^d} \int_{U_\ell}\int_{U_\ell}
|V^{1/p}(t)\vf(y)|^p\dt\dy\\
&=Lc'\sum_{\ell\in \bZ^d} \int_{U_\ell}\int_{U_\ell}
|V^{1/p}(t)V^{-1/p}(y)V^{1/p}(y)\vf(y)|^p\dt\dy\\
&\leq Lc'\sum_{\ell\in \bZ^d}\int_{U_\ell}\bigg( \int_{U_\ell}
\|V^{1/p}(t)V^{-1/p}(y)\|^p\dt\bigg) |V^{1/p}(y)\vf(y)|^p\dy\\
&= L|U_0|c'\sum_{\ell\in \bZ^d}\int_{U_\ell}\bigg( \frac{1}{|U_\ell|}\int_{U_\ell}
\|V^{1/p}(t)V^{-1/p}(y)\|^p\dt\bigg) |V^{1/p}(y)\vf(y)|^p\dy\\
&\leq L|U_0|c'  [V]_{\mathbf{A}_p(\cB_A)} \sum_{\ell\in \bZ^d}\int_{U_\ell}|V^{1/p}(y)\vf(y)|^p\dy\\
&\leq LN|U_0| c_M'  [W]_{\mathbf{A}_p(\cB_A)}\|\vf\|_{L^p(V)}^p\\
&=C^p \|\vf\|_{L^p(V)}^p,
\end{align*}
with $C^p:=Ln_0|U_0| c'[V]_{\mathbf{A}_p(\cB_A)}$, where we used \eqref{eq:height} for the final inequality. This completes the proof of the lemma in the case  $0<p\leq 1$. 

Let us now consider $1<p<\infty$ and solely focus on the modifications needed from the previous case. The same type of estimate as used in \eqref{eq:esti2}, and averaging over $u\in U_0$, leads to
\begin{align*}
|V^{1/p}(t)(\mathcal{F}^{-1}\phi * \vf)(t)|&\leq
\frac{c}{|U_0|}\sum_{\ell\in \bZ^d} \brac{t-\ell}_A^{-M} \int_{U_\ell}
|V^{1/p}(t)\vf(y)|\dy,
\end{align*}

By writing $M=M/p+M/p'$, and applying the discrete H\"older inequality to this estimate, we obtain
\begin{align*}
|V^{1/p}(t)(\mathcal{F}^{-1}\phi * \vf)(t)|^p&\leq 
 \frac{c'}{|U_0|}\sum_{\ell\in \bZ^d}  \bigg(\int_{U_\ell}
|V^{1/p}(t)\vf(y)|\dy\bigg)^p\brac{t-\ell}_A^{-M},
\end{align*}
where we have used the estimate \eqref{eq:integral} together with the assumption that $M>\nu$. We have, using H\"older's inequality,
\begin{align*}
    \bigg(\int_{U_\ell}
|V^{1/p}(t)\vf(y)|\dy\bigg)^p&\leq \bigg(\int_{U_\ell}
\|V^{1/p}(t)V^{-1/p}(y)\|\cdot|V^{1/p}(y)\vf(y)|\dy\bigg)^p\\
&\leq \bigg(\int_{U_\ell}
\|V^{1/p}(t)V^{-1/p}(y)\|^{p'}\dy\bigg)^{p/p'}
\int_{U_\ell}
|V^{1/p}(y)\vf(y)|^p\dy.
\end{align*}
Now, by Lemma \ref{le:do}, $v_x(y):=\|V^{1/p}(t)V^{-1/p}(y)\|^{p'}\in A_{p}(\cB_A)$ is doubling, so similar to estimate \eqref{eq:doub}, we obtain for $\beta=\nu p$,
\begin{equation}\label{eq:doub2}
\int_{U_\ell} \|V^{1/p}(t)V^{-1/p}(y)\|^{p'}\,dy\leq c_w\brac{k-\ell}_A^\beta\int_{U_k} \|V^{1/p}(t)V^{-1/p}(y)\|^{p'}\,dy.
\end{equation}
This brings us to the estimate, 
\begin{align*}
\|\phi(D)\vf\|_{L^p(V)}^p&\leq \sum_{k\in \bZ^d} \int_{U_k} |V^{1/p}(t)(\phi(D)\vf)(t)|^p\dt\\
&\leq \frac{c'}{|U_0|} \sum_{k\in \bZ^d}\sum_{\ell\in \bZ^d} \int_{U_k}\bigg(\int_{U_\ell}
\|V^{1/p}(t)V^{-1/p}(y)\|^{p'}\,dy\bigg)^{p/p'}\dt\\
&\qquad\qquad\times \brac{k-\ell}_A^{-M}\int_{U_\ell}
|V^{1/p}(y)\vf(y)|^p\dy\\
&\leq c'|U_0|^{p/p'} \sum_{k\in \bZ^d}\sum_{\ell\in \bZ^d} \bigg[\int_{U_k}\bigg(\int_{U_k}
\|V^{1/p}(t)V^{-1/p}(y)\|^{p'}\,\frac{\d y}{|U_k|}\bigg)^{p/p'}\,\frac{\dt}{|U_k|}\bigg]\\
&\qquad\qquad\times \brac{k-\ell}_A^{-M+\beta\frac{p}{p'}}\int_{U_\ell}
|V^{1/p}(y)\vf(y)|^p\dy\\
&\leq c'|U_0|^{p/p'} [V]_{\mathbf{A}_p(\cB_A)}^p \sum_{\ell\in \bZ^d}\int_{U_\ell}
|V^{1/p}(y)\vf(y)|^p\dy\\
&\leq  C^p\|\vf\|_{L^p(V)}^p,
\end{align*}
with $C^p:=c'n_0|U_0|^{p/p'} [V]_{\mathbf{A}_p(\cB_A)}^p$, where we have used \eqref{eq:Roudenko}, the assumption that $M-\beta\frac{p}{p'}=M-\beta(p-1)>\nu$, and $\eqref{eq:height}$ for the final estimate. This completes the proof of the lemma.
\end{proof}

\begin{remark}
In the isotropic setup with $A=\text{Id}_{d\times d}$, the author proved a version of Proposition \ref{prop:main} in \cite{Niel:2025a} valid in the restricted case $0<p<1$. 
\end{remark}

Next we consider a related sampling result in the matrix-weighted case. The result provides a generalization of a result obtained in the isotropic setup by Roudenko \cite{Roud:2003a} and Frazier and Roudenko \cite{FrazRoud:2004a}. The sampling grid appearing in the result is (in general) nonuniform, reflecting an adaptation to the geometry specified by the one-parameter dilation group.  

\begin{proposition}\label{prop:samp}
    Let $0<p<\infty$ and suppose $W\in\mathbf{A}_p(\cB_A)$. Let $B=B_A(c,R)\in\cB_A$,  put $T=\delta_{R}\cdot+c$. 
Then there exists a constant $c_{p,d}$ such that for $\vg\in E_B\cap L^p(W)$,
$$\sum_{\ell\in \bZ^d}\int_{U_{(B,\ell)}}\big|W^{1/p}(x)\vg\big(\delta_{R}^{-1}\ell\big)\big|^p\,dx\leq c_{p,d}\|\vg\|_{L^p(W)}^p,$$
with $U_{(B,\ell)}:=\delta_{R}^{-1} U_\ell$, where $U_\ell$ is defined in \eqref{eq:Qk}.
\end{proposition}
\begin{proof}
 Given $\vg\in E_B\cap L^p(W)$, and as noted earlier in \eqref{eq:vg}, the modified function 
 \begin{equation}\label{eq:unim}
 \tilde{\vg}(x):=e^{-ix\cdot \delta_{R}^{-1}c}\vg\big(\delta_{R}^{-1}x\big)
 \end{equation}
 satisfies that $\text{supp}(\hat{{\tilde{\vg}}})\subseteq B_0:=B_A(0,1)$.
Now, by a change of variable, 
\begin{align*}\int_{U_{(B,\ell)}}\big|W^{1/p}(x){\tilde{\vg}}\big(\ell\big)\big|^p\,dx&\asymp {R}^{-\nu}\int_{U_\ell}\bigg|W^{1/p}\bigg(\delta_{R}^{-1} u\bigg){\tilde{\vg}}\big(\ell\big)\bigg|^p\,du
\end{align*}
Let us suppose there exists a constant $c_{p,d}:=c_{p,d}([W(\delta_{R}^{-1}\cdot)]_{\mathbf{A}_p(\cB_A)})$ such that for
$\tilde{\vg}\in E_{B_0}$,
\begin{equation}\label{eq:fund}
\sum_{\ell\in \bZ^d}\int_{U_\ell}\bigg|W^{1/p}\bigg(\delta_{R}^{-1}x\bigg){\tilde{\vg}}\big(\ell\big)\bigg|^p\,dx\leq  c_{p,d}\|{\tilde{\vg}}\|_{L^p(W(\delta_{R}^{-1} \cdot))}^p.
\end{equation}
We then apply another change of variable, using that the scalar function $e^{-ix\cdot \delta_{R}^{-1}c}$ in \eqref{eq:unim} is unimodular,
\begin{align*}
\sum_{\ell\in \bZ^d}\int_{U_{(B,\ell)}}\big|W^{1/p}(x)\vg\big(\delta_{R}^{-1}\ell\big)\big|^p\,dx&\asymp
{R}^{-\nu}\sum_{\ell\in \bZ^d}\int_{U_\ell}\big|W^{1/p}\big(\delta_{R}^{-1}x\big){\tilde{\vg}}\big(\ell\big)\big|^p\,dx\\&\leq  c_{p,d}{R}^{-\nu}\|{\tilde{\vg}}\|_{L^p(W(\delta_{R}^{-1}\cdot))}^p\\
&=c_{p,d}\|\vg\|_{L^p(W)}^p.
\end{align*}
Now, by Lemma \ref{le:inv}, $[W(\delta_{R}^{-1}\cdot)]_{\mathbf{A}_p(\cB_A)}=[W]_{\mathbf{A}_p(\cB_A)}$ and $c_{p,d}$ is therefore independent of $B$. Hence, the proof has been reduced to proving the validity of \eqref{eq:fund}. Let $\vf\in E_{B_0}$ and take a scalar function $\phi\in\dS$ satisfying $\hat{\phi}|_{B_0}\equiv 1$, while $\text{supp}(\hat{\phi})\in \{z\in\bR^d:|z|<3\}$. Clearly, $\vf=\vf*\phi$ and we obtain the sampling representation, see \cite[Lemma 6.10]{FrazJaweWeis:1991a},
$$\vf(x)=\sum_{\ell\in \bZ^d} \vf(\ell)\gamma(x-\ell).$$
However, for any $u\in\bR^d$, $\vf^u(x):=\vf(x+u)\in E_{B_0}$, so the sampling representation  with $x=k-u$ yields
$$\vf^u(k-u)=\vf(k)=\sum_{\ell\in \bZ^d} \vf(\ell+u)\gamma(k-u-\ell).$$
Hence,
$$W^{1/p}(\delta_{R}^{-1}t)\vg(k)=\sum_{\ell\in \bZ^d} W^{1/p}(\delta_{R}^{-1}t)\vg(\ell+u)\gamma(k-u-\ell),$$
so
\begin{align*}
|W^{1/p}(\delta_{R}^{-1}t)\vg(k)|&\leq \sum_{\ell\in \bZ^d} |W^{1/p}(\delta_{R}^{-1}t)\vg(\ell+u)|\cdot |\gamma(k-u-\ell)|\\
&\leq  c_M\sum_{\ell\in \bZ^d} |W^{1/p}(\delta_{R}^{-1}t)\vg(\ell+u)|\cdot \brac{k-u-\ell}_A^{-M}
\end{align*}
for any $M>0$. The proof of \eqref{eq:fund} can now be completed by using Lemma \ref{le:help1} with $V=W(\delta_{R}^{-1}\cdot)$.
\end{proof}

\begin{lemma}\label{le:help1}
Let $V\in \mathbf{A}_p(\cB_A)$ for some $0<p<\infty$. Suppose for $\vg\in E_{B_A(0,1)}$ we have a uniform estimate
    \begin{align*}
|V^{1/p}(t)\vg(k)|&\leq c_M \sum_{\ell\in \bZ^d} |V^{1/p}(t)\vg(\ell+u)|\cdot \brac{k-u-\ell}_A^{-M},
\end{align*}
for any $M>0$, $k\in\bZ^d$, and $u\in\bR^d$. Then 
$$\sum_{\ell\in \bZ^d}\int_{U_{\ell}}\big|V^{1/p}(x)\vg\big(\ell\big)\big|^p\dx\leq c_{p,d}\|\vg\|_{L^p(V)}^p.$$
\end{lemma}
The proof of Lemma \ref{le:help2} can easily be adapted to provide a proof of Lemma \ref{le:help1}. We leave the details to the reader.

\section{Anisotropic matrix-weighted Besov spaces}\label{sec:B}
In this section we present an application of the Fourier multiplier and sampling result from the previous section. Specifically, we build anisotropic matrix-weighted Besov spaces providing a generalization of the matrix-weighted Besov spaces introduced by Roudenko \cite{Roud:2003a} and Frazier and Roudenko \cite{FrazRoud:2004a}. 

We first review some  known results on structured coverings of $\bR^d$ by  anisotropic balls that will be needed for the construction. We refer the reader to \cite{BoruNiel:2008a} for further details and proofs. 

Let $|\cdot|_A$ be a quasi-distance as in Definition \ref{def:A} associated with the one-parameter group of dilations $\delta_t$, $t>0$. 
For positive $c>0$, consider the family $\cC_c:=\big\{B_A\big(\zeta,c\brac{\zeta}_A\big)\big\}$. 
One can verify that there is a constant $0<\eta_0\leq 1$ such that for any $0<c\leq \eta_0$ there 
exists a  corresponding countable family
$\cC_c:=\big\{B_A\big(\zeta_j,c\brac{\zeta_j}_A)\big)\big\}_{j\in\bN}$ satisfying
\begin{itemize}
    \item $\cC_c$ covers $\bR^d$ and is of finite height, i.e., for some fixed $L:=L(c)> 1$,
     $$1\leq \sum_{j=1}^\infty \mathbf{1}_{B_A(\zeta_j,c\brac{\zeta_j})}(x)\leq L,\qquad x\in\bR^d$$
    \item there exists $c':=c'(c)<c$ such that the scaled sets $\big\{B_A\big(\zeta_j,c'\brac{\zeta_j}_A)\big)\big\}_{j\in\bN}$ are \textit{pairwise disjoint}.
\end{itemize}
Given $\cC_c$, one can always define an associated family $\{T_j\}$ of affine transformations by letting $T_j:=\delta_{t_j}\cdot+\zeta_j$ with $t_j:=\brac{\zeta_j}_A$ generating the covering, $\cC_c=\{T_jB_A(0,c)\}_j$. We refer to such $\cC_c$ as a \textit{structured covering} of $\bR^d$.

For any two  structured coverings $\cC_{\beta_1}$ and $\cC_{\beta_1}$, 
\begin{itemize}
    \item there exists a constant $C:=C(\beta_1,\beta_2)$ such that for any $B\in \cC_{\beta_1}$,
     \begin{equation}\label{eq:FIP}
     \#\{P\in \cC_{\beta_2}: B\cap P\not=\emptyset\}\leq C
     \end{equation}
\item     there is a constant $R=R(\beta_1,\beta_2)\geq 1$ such that whenever $B\in\cC_{\beta_1}$ and $P\in\cC_{\beta_2}$ satisfy $B\cap P\not=\emptyset$, we have
     \begin{equation}\label{eq:ratio}
     R^{-1}\leq \frac{\brac{\zeta}_A}{\brac{\gamma}_A}\leq R,\qquad \zeta\in B, \gamma\in P.
     \end{equation}
\end{itemize}

All the coverings $\cC_c$  in fact have a dyadic structure in the sense that it can be verified that there exists a  uniform $M:=M(c)>0$ such that  
$$\#\big\{\{B:B\in\cC_c\}: B \cap \{\zeta\in\bR^d: 2^j\leq |\zeta|_A<2^{j+1}\}\not=\emptyset\big\}\leq M,\qquad j\in\bN.$$
We refer to \cite[Section 4]{BoruNiel:2008a} for further details on structured coverings.

We say that $\Phi_{\cC_c}:=\{\phi_j\}_{j\in\bN}$ is a \textit{bounded  admissible partition of unity} (BAPU) subordinate to $\cC_c=\big\{B_A\big(\zeta_j,c\brac{\zeta_j}_A)\big)\big\}_{j\in\bN}$ if it is a $C^\infty$ resolution of the identity on $\bR^d$,  i.e., $\sum_j \phi_j(x)\equiv 1$ on $\bR^d$, such that  $\text{supp}(\phi_j)\subseteq B_A(\zeta_j,c\brac{\zeta_j}_A)$, $j\in\bN$. We also suppose that for any $M>0$ there exists $C:=C(\Phi,M,c)$ such that with $\tau_j:=\brac{\zeta_j}_A$,
\begin{equation}\label{eq:dec}|\mathcal{F}^{-1}\phi_j(x)|\leq C\tau_j^\nu(1+\tau_j|x|_A)^{-M},\qquad j\in\bN.
\end{equation}

Let us now pick $c_0$ such that $0<c_0\leq \eta_0/2$ and put
$c_1:=2c_0\leq \eta_0\leq 1$. Let
\[
\cC_{c_0}=\big\{B_A\big(\xi_j,c_0\brac{\xi_j}_A)\big)\big\}_{j\in\bN}
\]
be an associated covering. We let
\begin{equation}\label{eq:cF}
\cC:=\big\{B_A\big(\xi_j,c_1\brac{\xi_j}_A)\big)\big\}_{j\in\bN}.
\end{equation}

A BAPU can easily be constructed for $\mathcal{F}$. Pick a nonnegative $g\in
C^\infty(\bR^d)$  with $\supp(g)\subset {B}_A(0,2c_0)$ and $g(\xi)=1$
for $\xi\in {B}_A(0,c_0)$. Then, with $T_j:=\delta_{t_j}\cdot+\xi_j$, where $t_j:=\brac{\xi_j}_A$, 
we let  
\begin{equation}
  \label{eq:bapu_def}
  \phi_j:=\frac{g(T_j^{-1}\cdot)}{\sum_{k\in\bN} g(T_k^{-1} \cdot)}\in\dS,
\end{equation}
which  defines an
associated BAPU. In a similar fashion, 
 \begin{equation}
   \label{eq:bapusq}
   \psi_j:=
\frac{g(T_j^{-1}\cdot)}{\sqrt{\sum_{k\in\bN} g^2(T_k^{-1} \cdot)}}\in \dS,
 \end{equation}
 defines a smooth
``square root'' of the BAPU in the sense that $\sum_j \psi_j^2(x)\equiv 1$. The system $\{\psi_j\}$ also satisfies 
\eqref{eq:dec}. We again refer to \cite{BoruNiel:2008a} for further details.

We now follow the approach in Roudenko \cite{Roud:2003a} and Frazier and Roudenko \cite{FrazRoud:2004a} and give the following definition of anisotropic matrix-weighted Besov spaces, where it will be convenient to consider the direct sum Fr\'{e}chet space $\bigoplus_{j=1}^N\dS$, with dual space $\bigoplus_{j=1}^N\ddS$ consisting of $N$-tuples of tempered distributions.

\begin{definition}\label{def:Besov}
Let $W:\bR^d\rightarrow \bC^{N\times N}$ be a matrix weight. Let $\cC=\{P_k\}_{k\in\bN}$ be the covering defined in \eqref{eq:cF} and $\Phi=\{\phi_j\}_j$ be the BAPU given in Equation~\eqref{eq:bapu_def}. We define the inhomogeneous
matrix-weighted Besov space ${B}_{p,q}^s(A,W):={B}_{p,q}^s(A,W,\Psi)$ for $0<q	\leq \infty$ as the collection of vector-valued tempered distributions $\vf=(f_1,\ldots f_N)^\top\in \bigoplus_{j=1}^N\mathcal{S}'(\bR^d)$ satisfying
\begin{equation}\label{eq:defB}
\|\vf\|_{{B}_{p,q}^s(A,W)}:=\bigg(\sum_{j=1}^\infty |P_j|^{sq/\nu}\|\phi_j(D)\vf\|_{L^p(W)}^q\bigg)^{1/q}<+\infty,
\end{equation}
with the sum replaced by $\sup_j$ in the case $q=\infty$.
\end{definition}

\begin{remark}\leavevmode 
\begin{itemize}
\item In Definition \ref{def:Besov} we assume that $\vf$ is a vector-valued tempered distribution, while $\phi_j\in\dS$, so $\phi_j(D)\vf$ is in fact always a $C^\infty$ vector-valued function, making it feasible to actually test the condition given in \eqref{eq:defB}. 
    \item We have $|P_j|\asymp t_j^\nu$ uniformly in $j$, so we may replace $|P_j|^{1/\nu}$ by $t_j$ in \eqref{eq:defB} to get an equivalent (quasi-)norm whenever it is convenient.
\end{itemize}
\end{remark}

In the isotropic case, $A=\text{Id}_{d\times d}$, Definition \ref{def:Besov} 
provides the exact same spaces as studied in \cite{Roud:2003a,FrazRoud:2004a}, where the authors derives a significant number of interesting properties of the spaces provided that the matrix weight satisfies a suitable $A_p$-condition.  The purpose here is to show that the various invariance results derived in the previous sections allow one to extend the theory to the anisotropic setting. Let us first state the most fundamental result on the spaces from Definition \ref{def:Besov}.

\begin{proposition}\label{prop:complete}
Let $0< p<\infty$ and suppose $W\in \mathbf{A}_p(\cB_A)$. For $0<q\leq \infty$ and $s\in\bR$, 
\begin{itemize}
\item[(i)] We have continuous embeddings
$$\bigoplus_{j=1}^N\dS\hookrightarrow B^{s}_{p,q}(A,W)\hookrightarrow \bigoplus_{j=1}^N\mathcal{S}'(\bR^n).$$
    \item[(ii)] The space $B^{s}_{p,q}(A,W)$ is complete, i.e., $B^{s}_{p,q}(A,W)$ is a (quasi-)Banach space.
    \item[(iii)] The space $B^{s}_{p,q}(A,W)$ is independent of the choice of BAPU (up to equivalence of norms).
\end{itemize}
\end{proposition}

\begin{proof}
    The proof of (i) and (ii) can be found in Appendix \ref{app:A}. For (iii), let $\Phi=\{\phi_k\}_{k\in\bN}$ be the BAPU defined in \eqref{eq:bapu_def} with corresponding structured covering $\cC=\{P_k\}_{k\in\bN}$ and suppose 
    $\cC_{\beta}=\{Q_\ell\}_{\ell\in\bN}$, with $Q_\ell=B_A(\zeta_\ell,c\brac{\zeta_\ell}_A)$, is another structured covering with associated BAPU, $\Gamma=\{\gamma_\ell\}_{\ell\in \bN}$.
We first notice, using uniformly bounded height of   structured covering, that for $k\in \bN$,
\begin{equation}\label{eq:pdo}
\phi_k(D)\vf=\phi_k(D)\sum_{\ell\in A_k}\gamma_\ell(D)\vf,
\end{equation}
with $A_k=\{\ell\in \bN:P_k\cap Q_\ell\not=\emptyset\}$, where  $\#A_k$ is bounded by a uniform constant $C$, see Equation~\eqref{eq:FIP}.  Hence, by Proposition \ref{prop:main},
$$\|\phi_k(D)\vf\|_{L^p(W)}\leq C\sum_{\ell\in A_k}\|\gamma_{\ell}(D)\vf\|_{L^p(W)}.$$
By Equation~\eqref{eq:ratio}, for $Q_\ell$ and $P_k$ with  $P_k\cap Q_\ell\not=\emptyset$, we have $|Q_k|\asymp |P_\ell|$ uniformly in $k$ and $\ell$.  It follows from this observation that
$$|P_k|^{s/\nu}\|\phi_k(D)\vf\|_{L^p(W)}\leq C\sum_{\ell\in A_k}|P_\ell|^{s/\nu}\|\gamma_{\ell}(D)\vf\|_{L^p(W)}.$$
Similarly, we obtain, for $\ell\in\bN$,
$$|Q_\ell|^{s/\nu}\|\gamma_\ell(D)\vf\|_{L^p(W)}\leq C\sum_{k\in B_\ell}|P_k|^{s/\nu}\|\phi_k(D)\vf\|_{L^p(W)},$$
with $B_\ell=\{k\in \bN:Q_\ell\cap P_k\not=\emptyset\}$.
Using the uniform bounds on the cardinality of the sets $A_k$ and $B_\ell$, it is  then straightforward to verify that
\begin{align*}
\|\vf\|_{B^{s}_{p,q}(A,W)}&=\bigg\|\bigg\{|P_k|^{s/\nu}\|\phi_k(D)\vf\|_{L^p(W)}\bigg\}_k\bigg\|_{\ell_q}\\&\asymp 
\bigg\|\bigg\{|Q_j|^{s/\nu}\|\gamma_\ell(D)\vf\|_{L^p(W)}\bigg\}_\ell\bigg\|_{\ell_q},\qquad \vf\in 
B^{s}_{p,q}(A,W).
\end{align*}
\end{proof}

\begin{remark}\label{rem:en}
    The same argument as  used in the proof above applied to the ``square root'' BAPU $\Psi=\{\psi_j\}_j$ defined in \eqref{eq:bapusq} provides us with the norm equivalence
    \begin{equation}\label{eq:equivB}
\|\vf\|_{{B}_{p,q}^s(A,W)}\asymp \bigg(\sum_{j=1}^\infty |P_j|^{sq/\nu}\|\psi_j(D)\vf\|_{L^p(W)}^q\bigg)^{1/q}<+\infty,
\end{equation}
provided $W\in \mathbf{A}_p(\cB_A)$, $0<p<\infty$. We leave the details to the reader.
\end{remark}

\subsection{Tight frames for matrix-weighted Besov spaces}
A tight frame associated with the covering $\cC$ can easily be constructed using the system $\Psi$ defined in \eqref{eq:bapusq}.   Recall  that $c_1\leq 1$, so it follows from \eqref{eq:scaling} that $${B}_A(0,c_1)\subseteq \{z\in\bR^d:|z|<2\}\subset [-\pi,\pi)^d.$$ 
As before, put 
$t_j:=\brac{\xi_j}_A$ with $\xi_j$ from $\cC$. We  define
$$e_{k,\ell}(\xi) :=
(2\pi)^{-\frac{d}2}t_k^{-\nu/2}\mathbf{1}_{[-\pi,\pi)^d}(T_k^{-1}\xi)e^{i\ell\cdot
  T_k^{-1}\xi},\qquad \ell\in \bZ^d,\; k\in \bN,$$
and define the functions $\omega_{k,\ell}$ in the Fourier domain as follows
\begin{equation}
  \label{eq:ltf}
  \hat\omega_{k,\ell} :=\psi_ke_{k,\ell}\qquad \ell\in \bZ^d,\; k\in \bN.
\end{equation}
It is easy to verify that
$\Omega:=\{\omega_{k,\ell}\}_{k,\ell}$ is a tight frame for $L_2(\bR^d)$. In face, we just need
to notice that  $\{e_{k,\ell}\}_{\ell\in \bZ^d}$ is an orthonormal basis for
  $L^2(T_k([-\pi,\pi)^d))$ with $\supp(\psi_k)\subset T_kB_A(0,c_1)\subset T_k([-\pi,\pi)^d)$, which yields
  $$
  \sum_{\ell\in\bZ^d} |\langle f,\omega_{k,\ell}\rangle|^2=\sum _{\ell\in\bZ^d} |
  \langle \psi_{k}\hat{f},e_{k,\ell}\rangle|^2 =\|\psi_{k}
  \hat{f}\|_{L_2}^2.
  $$ 
  Moreover, as $\{\psi^2_k\}_{k\in \bN}$ forms a partition of unity,
  \begin{equation}\label{eq:frame}\sum_{\ell\in\bZ^d,k\in\bN} |\langle f,\omega_{k,\ell}\rangle|^2=
  \sum_{k\in \bN} \|\psi_{k}\hat{f}\|_{L_2}^2 =
  \int_{\bR^d} \sum_{k\in \bN} \psi_k^2(\xi)|\hat{f}(\xi)|^2\, \d\xi
  =\|f\|_{L_2}^2.
  \end{equation}

We
can also obtain an explicit representation of $\omega_{k,\ell}$ in direct
space. Put $\hat{\mu}_k(\xi):= \psi_k(T_k\xi)$. Then
\begin{equation}
  \label{eq:ltfd}
  \omega_{k,\ell}(x) = (2\pi)^{-\frac{d}2}t_k^{\nu/2} \mu_k(
    \delta_{t_k}x+\ell) e^{ix\cdot \xi_k}.
\end{equation}
It can be
verified, using the uniform compact frequency support of the functions $\{{\mu}_k\}$, that
$$|\mu_k(x)| \leq C_M \brac{x}_A^{-M},$$
for any $M\in \bN$, with $C_M$ independent of $x\in \bR^d$ and $k\in \bN$. 
Hence,
\begin{equation}
  \label{eq:dec_eta}
  |\omega_{k,\ell}(x)| \leq Ct_k^{\nu/2} \left(1+{t_k}\big|x-x_{k,\ell}\big|_A\right)^{-M},\qquad \text{where } x_{k,\ell}:=\delta_{t_k}^{-1}\ell.
\end{equation}
We  again refer to \cite{BoruNiel:2008a} for the details.

Let us now introduce suitable discrete sparseness spaces with the goal to obtain a discrete characterization of smoothness measured in $B^s_{p,q}(A,W)$.  Recall that for $P_k\in\cC$, we have $P_k=\delta_{t_k}B_A(0,c_1)+\xi_k$. For $k\in\bN$, $\ell\in\bZ^d$, we define 
\begin{equation}\label{eq:uk}
U{(k,\ell)}:=\delta_{ t_k}^{-1} U_\ell=B_A\left(x_{k,l},t_k^{-1}r_0\right),
\end{equation}
with $U_\ell$ defined in \eqref{eq:Qk} and $x_{k,l}$ given in \eqref{eq:dec_eta}.  Notice that uniformly in $k$, 
\begin{equation}\label{eq:U_kl}
|U{(k,\ell)}|=t_k^{-\nu}|U_\ell|=t_k^{-\nu}|U_0|\asymp t_k^{-\nu}.
\end{equation}
We now use the sets $\{U(k,\ell)\}$ to define the following vector-valued sequence spaces.
\begin{definition}
Let $W:\bR^d\rightarrow \bC^{N\times N}$ be a matrix weight, and suppose
 $s\in\bR$, $0< p<\infty$, and $0<q\leq \infty$. We let 
let $\cU=\{U(k,\ell)\}_{k\in\bN,\ell\in\bZ^d}$ be the collection of sets defined in \eqref{eq:uk}.
Let $b^{s}_{p,q}(A,W)$ denote the collection of all vector-valued sequences 
$\vs=\{\vs_{(k,\ell)}\}_{(k,\ell)\in\bN\times\bZ^d}$, where $\vs_{(k,\ell)}=\big(s_{(k,\ell)}^{(1)}, \ldots,s_{(k,\ell)}^{(N)}\big)^\top$, such that
\begin{align*}
\|\vs\|_{b^{s}_{p,q}(A,W)}&:=\bigg\|\bigg\{t_k^s\bigg\|\sum_{\ell\in\bZ^d} |U(k,\ell)|^{-\frac{1}{2}}\vs_{(k,\ell)}\mathbf{1}_{U(k,\ell)}\bigg\|_{L^p(W)}\bigg\}_k\bigg\|_{\ell_q}\\
&=\bigg(\sum_{k\in\bN}\bigg\| r_k^s\sum_{\ell\in\bZ^d} |U(k,\ell)|^{-\frac{1}{2}}\big\|W^{1/p}(t)\vs_{(k,\ell)}\big\|\mathbf{1}_{U(k,\ell)}(t)\bigg\|_{L^p(\dt)}^q\bigg)^{1/q}.
\end{align*}
 For $q=\infty$, the $\ell^q$-norm is replaced by the supremum over $k$.
\end{definition}

The sampling result in Proposition \ref{prop:samp} can now be used to obtain the following estimate for the coefficient operator $C$
defined on $B_{p,q}^{s}(A,W)$ by
  \begin{equation}\label{eq:coefficient}
C\vf:=\big(\langle \vf,\omega_{k,\ell}\rangle\big)_{k\in\bN,\ell\in\bZ^d},\qquad \vf\in B_{p,q}^{s}(A,W).
\end{equation}

\begin{proposition}\label{prop:coef}
  Let  $0< p<\infty$ and $0<q\leq\infty$. Suppose $W\in \mathbf{A}_p(\cB_A)$ and let $\Omega=\{\omega_{k,\ell}\}_{k,\ell}$ be the system defined in \eqref{eq:ltf}. Then there exists a constant $C:=C(p,q,W)$ such that for $\vf\in B_{p,q}^{s}(A,W)$,
  \begin{equation}\label{eq:coe}
      \|\{\vc_{k,\ell}\}_{k,\ell}\|_{b_{p,q}^{s}(A,W)}\leq C \|\vf\|_{B_{p,q}^{s}(A,W)},
  \end{equation}
  with $\vc_{k,\ell}:=\langle \vf,\omega_{k,\ell}\rangle$.
\end{proposition}
\begin{proof}
Recall that for  $T_k:=\delta_{t_k}\cdot+\xi_k$,
$\omega_{k,\ell}\in E_{T_k B_A(0,c_1)}$ with
$$\hat{\omega}_{k,\ell}(\xi)=(2\pi)^{-\frac{d}2}t_k^{-\nu/2}\mathbf{1}_{[-\pi,\pi)^d}(T_k^{-1}\xi)e^{i\ell\cdot
  T_k^{-1}\xi}\psi_k(\xi).
$$
We have $T_k^{-1}=\delta_{t_k}^{-1}(\cdot-\xi_k)$, which implies, using \eqref{eq:ltf},
$$|\langle \vf,\omega_{k,\ell}\rangle|=c|U(k,\ell)|^{1/2}[\psi_k(D)\vf]\left(\delta_{t_k}^{-1}\ell\right),$$
with $c:=(2\pi)^{-d/2}|U_0|^{-1}$  a fixed constant. We notice that $$\psi_k(D)\vf\in E_{T_kB_A(0,c_1)}\subseteq E_{T_kB_A(0,1)}=E_{B_A(\xi_k,t_k)}.$$ 

Now, as the sets $U(k,\ell)$ satisfy the exact same finite height property \eqref{eq:height} as the unit-scale sets $\{U_\ell\}$,  we use the general sampling in Proposition \ref{prop:samp} with $B=B_A(\xi_k,t_k)$ to deduce that there exists a constant $c_{p,\nu}:=c_{p,\nu}([W]_{\mathbf{A}_p(\cB_A)})$, such that
\begin{align*}
  \|  \{\langle \vf,\omega_{k,\ell}\rangle\}_{k,\ell} \|_{b_{p,q}^{s}(A,W)}&=
  \bigg\|\bigg\{t_k^s\bigg\|\sum_{\ell\in\bZ^d} |U(k,\ell)|^{-\frac{1}{2}}\langle \vf,\omega_{k,\ell}\rangle\mathbf{1}_{U(k,\ell)}\bigg\|_{L^p(W)}\bigg\}_k\bigg\|_{\ell_q}\\
  &\asymp\bigg\|\bigg\{t_k^s\bigg[\sum_{\ell\in \bZ^d}\int_{U(k,\ell)}\bigg|W^{1/p}(x)[\psi_k(D)\vf] \left(\delta_{t_k}^{-1}\ell\right)\bigg|^p \dx\bigg]^{1/p}\bigg\}_k\bigg\|_{\ell_q}\\
    &\leq c_{p,\nu}\left\|\left\{t_k^s\|\psi_k(D)\vf\|_{L^p(W)} \right\}_k\right\|_{\ell_q}\\
&\asymp  c_{p,\nu}\|\vf\|_{B_{p,q}^{s}(A,W)},    
\end{align*}
where we used the observation in Remark \ref{rem:en} for the final estimate.
\end{proof}

  We turn our attention to the canonical reconstruction/synthesis operator $T$ for the tight frame $\{\omega_{k,\ell}\}$.  The reconstruction/synthesis operator $T$ is defined for (finite) sequences $\vc=\{\vc_{k,\ell}\}$ by 
  \begin{equation}\label{eq:reconstruct}
  T\vc:=\sum_{{k,\ell}}\vc_{k,\ell} \omega_{k,\ell}.
  \end{equation}
  We will prove that $T$  is bounded from $b_{p,q}^{s}(A,W)$ to $B_{p,q}^{s}(A,W)$ for suitable weights $W$.
\begin{remark}\label{rem:T}
The fact that $\{\omega_{k,\ell}\}$ is a redundant system in $L^2(\mathbb{R}^d)$ makes it a non-trivial issue to verify that $T$ is actually well-defined with an extension to all of $b_{p,q}^{s}(A,W)$ based on the result for finite sequences presented in Proposition \ref{prop:recon} below. We shall not pursue the details here, but refer the reader to the general approach to the issue presented in \cite[Section 3]{GeorJohnNiel:2017a}, which is applicable to the present setup.
\end{remark}  
   We have the following result that in particular applies to matrix weights in $\mathbf{A}_p(\cB_A)$, cf.\ (i) in Lemma \ref{le:do}. 
\begin{proposition}\label{prop:recon}
  Let  $0< p<\infty$, $0<q<\infty$, and let $W$ be a matrix weight 
  for which $u_x(\cdot):=|W^{1/p}(\cdot)x|^p$  satisfies the doubling condition \eqref{eq:doub} uniformly in $x\in\bR^d$. Then there exists a constant $C$ such that for any finite vector-valued coefficient sequence $\vc:=\{\vc_{k,\ell}\}_{(k,\ell)\in F}$, $F\subset \bN\times\bZ^d$,
  \begin{equation}
      \bigg\|\sum_{(k,\ell)\in F}\vc_{k,\ell}\omega_{k,\ell}\bigg\|_{B_{p,q}^{s}(A,W)}\leq C \|\{\vc_{k,\ell}\}\|_{b_{p,q}^{s}(A,W)}.
  \end{equation}
  
\end{proposition}
\begin{proof}
Let $\Phi=\{\phi_k\}_k$ be the BAPU associated with $\cC=\{P_k\}_k$. We have, using the fact that $\text{supp}(\phi_k)\subseteq P_k=T_kB_A(0,c_1)$,

\begin{align*}
      \bigg\|\sum_{(j,\ell)\in F}\vc_{j,\ell}\omega_{j,\ell}\bigg\|_{B_{p,q}^{s}(A,W)}&=
      \bigg\|\bigg\{t_k^s\bigg\|\phi_k(D)\sum_{(j,\ell)\in F}\vc_{j,\ell}\omega_{j,\ell}\bigg\|_{L^p(W)}\bigg\}_k\bigg\|_{\ell_q}\\
      &= \bigg\|\bigg\{t_k^s\bigg\|\phi_k(D)\sum_{j\in N(k)}\sum_\ell\vc_{j,\ell}\omega_{j,\ell}\bigg\|_{L^p(W)}\bigg\}_k\bigg\|_{\ell_q},
\end{align*}
where $N(k)=\{m\in\bN:P_m\cap P_k\not=\emptyset\}$. Now, by \eqref{eq:ratio}, $t_j\asymp t_k$ (uniformly) for $j\in N(k)$, so by Proposition \ref{prop:main},

\begin{align*}
t_k^s\bigg\|\phi_k(D)\sum_{j\in N(k)}\sum_\ell\vc_{j,\ell}\omega_{j,\ell}\bigg\|_{L^p(W)}
     & \leq C t_k^s\bigg\|\sum_{j\in N(k)}\sum_\ell\vc_{j,\ell}\omega_{j,\ell}\bigg\|_{L^p(W)}\\
     &\leq C' \sum_{j\in N(k)}\bigg\|t_j^s \sum_\ell\vc_{j,\ell}\omega_{j,\ell}\bigg\|_{L^p(W)},
\end{align*}
where we relied on the localization property stated in \eqref{eq:dec}.
Recall that $\omega_{j,\ell}$ satisfies the decay property \eqref{eq:dec_eta} for any $M>0$. From this we obtain the estimate

\begin{align*}
\bigg\| \sum_\ell\vc_{j,\ell}\omega_{j,\ell}\bigg\|_{L^p(W)}^p&\leq 
\int_{\bR^d} \bigg(\sum_\ell |W^{1/p}(x)\vc_{j,\ell}|\, |\omega_{j,\ell}(x)|\bigg)^p\dx\\
&\leq C_M\int_{\bR^d} \bigg(t_j^{\nu/2}\sum_\ell |W^{1/p}(x)\vc_{j,\ell}|\, \big(1+t_j|x-x_{j,\ell}|_A\big)^{-M}\bigg)^p\dx\\
&\leq C_M't_j^{\nu p/2}\int_{\bR^d} \sum_\ell |W^{1/p}(x)\vc_{j,\ell}|^p\, \big(1+t_j\big|x-x_{j,\ell}\big|_A\big)^{-\frac{Mp}{2}}\dx,
    \end{align*}
where we used the discrete H\"older inequality for the last step in the case $1<p<\infty$ with $M$ chosen large enough such that for the dual H\"older exponent $p'$ to $p$,
$$
  \sup_{u\in\bR^d} \sum_\ell \left(1+\big|u-\ell\big|_A\right)^{-\frac{Mp'}{2}} <\infty.$$
  For $0<p\leq 1$, we may obtain the needed estimate directly using sub-additivity. 
  By assumption, the scalar function $w_{j,\ell}(x):=|W^{1/p}(x)\vc_{k,\ell}|^p$ is doubling and thus has a doubling exponent $\beta>0$ from \eqref{eq:dou_exp} independent of $j$ and $\ell$. We therefore use Lemma \ref{le:sq} below to obtain the following estimate,
  \begin{align*}
\bigg\| \sum_\ell\vc_{j,\ell}\omega_{j,\ell}\bigg\|_{L^p(W)}^p&\leq 
C_N't_j^{\nu p/2}\sum_\ell\int_{\bR^d}  |W^{1/p}(x)\vc_{j,\ell}|^p\, \big(1+t_j\big|x-x_{j,\ell}\big|_A\big)^{-\frac{Mp}{2}}\dx\\
&\leq 
C_N't_j^{\nu p/2}\sum_{\ell\in\bZ^d}\int_{U(j,\ell)}  |W^{1/p}(x)\vc_{j,\ell}|^p\dx\\
&\asymp\bigg\| \sum_\ell |U(j,\ell)|^{-1/2} \vc_{j,\ell} \mathbf{1}_{U(j,\ell)}\bigg\|_{L^p(W)}^p,
\end{align*}
where we have used \eqref{eq:U_kl}. 
We may now conclude that
\begin{align}
      \bigg\|\sum_{(k,\ell)\in F}\vc_{k,\ell}\omega_{k,\ell}\bigg\|_{B_{p,q}^{s}(A,W)}&\leq C
     \bigg\|\bigg\{  \sum_{j\in N(k)}\bigg\|t_j^s \sum_\ell\vc_{j,\ell}\omega_{j,\ell}\bigg\|_{L^p(W)}\bigg\}_k\bigg\|_{\ell_q}\nonumber\\
     &\leq 
 C  \bigg\|\bigg\{  \sum_{j\in N(k)}t_j^s\bigg\| \sum_\ell |U(j,\ell)|^{-1/2} \vc_{k,\ell} \mathbf{1}_{U(j,\ell)}\bigg\|_{L^p(W)}\bigg\}_k\bigg\|_{\ell_q}\nonumber\\  
   &\leq 
 C  \bigg\|\bigg\{  \sum_{k}t_k^s\bigg\| \sum_\ell |U(k,\ell)|^{-1/2} \vc_{k,\ell} \mathbf{1}_{U(k,\ell)}\bigg\|_{L^p(W)}\bigg\}_k\bigg\|_{\ell_q}\nonumber\\
 &\asymp \|\{\vc_{k,\ell}\}\|_{b_{p,q}^{s}(A,W)},\label{eq:LW}
 \end{align}
 where the uniform bound on the cardinality of $N(k)$ was used. This concludes the proof.
 \end{proof}

\begin{remark}\label{rem:dense}
Since $\{\omega_{k,\ell}\}_{k,\ell}\subset \mathcal{S}(\bR^d)$, it follows easily from Proposition  \ref{prop:coef} and Proposition \ref{prop:recon}, by standard arguments, that $\bigoplus_{j=1}^N\dS$ is dense in $B_{p,q}^{s}(A,W)$ whenever $0< p<\infty$, $0<q<\infty$, and $W\in\mathbf{A}_p(\cB_A)$.
\end{remark}

The following technical lemma was used in the proof of Proposition \ref{prop:recon}.
\begin{lemma}\label{le:sq}
Let $w:\bR^d \rightarrow (0,\infty)$ be a function satisfying the following doubling estimate, for $\lambda\geq 1$,
  \begin{equation*}\label{eq:doub3}
        \int_{B_A(x,\lambda r)}w(t)\dt\leq c\lambda^\beta \int_{B_A(x,r)} w(t)\dt,\qquad x\in\bR^d,r>0,
    \end{equation*}
 for some doubling exponent $\beta>0$. Let $j\in\bN$, $\ell\in\bZ^d$ and let the quantities $U(j,\ell)$, $t_j$, $x_{j,\ell}$ be defined as in Equations\ \eqref{eq:uk} and \eqref{eq:dec_eta}. Then for $L>\beta$, we have 
$$\int_{\bR^d}  w(x)\big(1+t_j\big|x-x_{j,\ell}\big|_A\big)^{-L}\dx \leq C\int_{U(j,\ell)}  w(x)\, \dx.$$
\end{lemma}
\begin{proof}
We make a partition $\bR^d=\bigcup_{m=0}^\infty R_m$, where $R_0=U(j,\ell)=B_A(x_{j,\ell},t_j^{-1}r_0)$ 
and the  ``annuli'' $R_m$, $m\geq 1$, is defined by
$$R_m:=\big\{y\in\bR^d: 2^{m-1}t_j^{-1}r_0\leq |y-x_{j,\ell}|_A <  2^{m}t_j^{-1}r_0\big\}.$$
Then 
\begin{align*}
 \int_{\bR^d}  w(x)\big(1+t_j\big|x-x_{j,\ell}\big|_A\big)^{-L}\dx&=
\sum_{m=0}^\infty \int_{R_m}  w(x) \big(1+t_j\big|x-x_{j,\ell}\big|_A\big)^{-L}\dx\\
&\leq C \sum_{m=0}^\infty 2^{-mL}\int_{R_m}  w(x)\, \dx\\
\end{align*}
However, by the doubling property of $w(x)$, noting that $R_m\subseteq \{y: |y-x_{j,\ell}|_A < 2^m t_j^{-1}r_0\}$, we have
$$\int_{R_m}  w(x)\dx \leq c'2^{m\beta}\int_{R_0}  w(x)\, \dx,$$
so
$$\int_{\bR^d}  w(x)\big(1+t_j\big|x-x_{j,\ell}\big|_A\big)^{-L}\dx \leq C \sum_{m=0}^\infty  2^{(\beta -L)m}\int_{R_0}  w(x)\, \dx\leq C'\int_{R_0}  w(x)\, \dx,$$
provided that $L>\beta$.
\end{proof}

\begin{remark}
As a consequence of the precise doubling condition in \eqref{eq:precise_doubling}, we may use $\beta=\max\{\nu,\nu p\}$ as doubling exponent in Lemma \ref{le:sq} for $w_x:=|W^{1/p}(\cdot)x|^p$ whenever we have $W\in \mathbf{A}_p(\cB_A)$, $0<p<\infty$.    
\end{remark}

We conclude the paper with the following corollary to Propositions \ref{prop:coef} and \ref{prop:recon}.

\begin{corollary}
    Let  $0< p<\infty$ and let $0<q<\infty$. Suppose $W\in \mathbf{A}_p(\cB_A)$ and let $\Omega=\{\omega_{k,\ell}\}_{k,\ell}$ be the system defined in \eqref{eq:ltf}. Then the coefficient operator $C$ given by \eqref{eq:coefficient} and
reconstruction operator $T$ given by \eqref{eq:reconstruct}
both extend to bounded operators and makes $B^s_{p,q}(A,W)$ a retract of
$b^s_{p,q}(A,W)$, i.e.,
$TC=\operatorname{Id}_{B^s_{p,q}(A,W)}$. Moreover, we have the norm equivalence
$$\|\vf\|_{B^s_{p,q}(A,W)}\asymp \|C\vf\|_{b^s_{p,q}(A,W)},\qquad \vf\in B^s_{p,q}(A,W).$$
\end{corollary}
\begin{proof}
 Considering the results from Propositions \ref{prop:coef} and \ref{prop:recon}, all that remains is to address the claim that $TC=\operatorname{Id}_{B^s_{p,q}(A,W)}$. For this, we simply use the observation that $TC=\operatorname{Id}_{L^2(\bR^d)}$ since $\Omega$ is a tight frame, cf.\ \eqref{eq:frame}. As noted in Remark \ref{rem:dense}, $\dS\subset L^2(\bR^d)$ is dense in $B^s_{p,q}(A,W)$ so a simple extension argument then proves the claim.
\end{proof}

\appendix
\section{Completeness of Anisotropic matrix-weighted Besov spaces}\label{app:A}
Here we complete the proof of Proposition \ref{prop:complete}.
We recall that $\bigoplus_{j=1}^N \dS$ denotes the family of vector functions $\vf=(f_1,\ldots,f_N)^\top$ with $f_i\in\dS$, $i=1,\ldots,N$. We equip the space with the induced semi-norms
$$p_m(\vf):=\sum_{j=1}^N p_m(f_j),\qquad\text{with } p_m(f_i):=\sup_{\xi\in\bR^d}  \brac{\xi}_A^d\sum_{|\eta|\leq d} |\partial^\eta\hat{f}_i(\xi)|.$$
It follows from the growth estimate in Equation~\eqref{natasha} that the semi-norms define the usual topology on $\bigoplus_{j=1}^N \dS$.
As before, we let $\bigoplus_{j=1}^N \mathcal{S}'(\bR^d)$ denote the corresponding family of vector-valued tempered distributions.

We shall use the following elementary consequence of the scalar doubling estimate already used in the proof of Lemma~\ref{le:help2}. Let $0<p\leq 1$ and $W\in \mathbf{A}_p(\cB_A)$. By Lemma~\ref{le:sc}, the scalar weights
\[
    w_v(x):=|W^{1/p}(x)v|^p, \qquad v\in\bC^N,
\]
belong to $A_1(\cB_A)$ uniformly for $|v|=1$. In particular, using the weighted doubling estimate \eqref{eq:doubb2}, one may compare $U(k,\ell)$ with a fixed reference ball by means of the quasi-triangle inequality. Since
\[
    \inf_{|v|=1}\int_{B_A(0,r_0)} |W^{1/p}(x)v|^p\,\d x>0,
\]
by compactness of the unit sphere and the positive definiteness of $W$ a.e., the same comparison argument leading to \eqref{eq:doub} gives constants $C>0$ and $M_0>0$ such that
\begin{equation}\label{eq:scalar-lower-average}
    |v|\leq C t_k^{M_0}(1+|x_{k,\ell}|_A)^{M_0}|U(k,\ell)|^{-1/p}
    \left(\int_{U(k,\ell)} |W^{1/p}(x)v|^p\,\d x\right)^{1/p}
\end{equation}
for all $k\in\bN$, $\ell\in\bZ^d$, and $v\in\bC^N$. Indeed, by homogeneity it suffices to consider $|v|=1$, and the constants are uniform because the $A_1(\cB_A)$-bounds of $w_v$ are uniform in $v$.

We shall also use the following standard decay estimate for the frame coefficients of a Schwartz function. Since $\boldsymbol{\theta}$ is Schwartz and the system $\{\omega_{k,\ell}\}_{k,\ell}$ is uniformly localized, integration by parts gives, for every $L>0$,
\begin{equation}\label{eq:theta-frame-coefficients-decay}
\left|
\langle \boldsymbol{\theta},\omega_{k,\ell}\rangle
\right|
\leq C_L p_L(\boldsymbol{\theta})
t_k^{-L}(1+|x_{k,\ell}|_A)^{-L}.
\end{equation}
Consequently, choosing $L$ sufficiently large and using $\sum_k t_k^{-\nu-\epsilon}<\infty$ for every $\epsilon>0$, as follows from \eqref{eq:integral}, we have
\begin{equation}\label{eq:theta-weighted-summability}
\sum_{k\in\bN} t_k^{\nu(1/p-1/2)-s}
\sup_{\ell\in\bZ^d}
\left[t_k^{M_0}(1+|x_{k,\ell}|_A)^{M_0}
\left|\langle \boldsymbol{\theta},\omega_{k,\ell}\rangle\right|\right]
\leq C p_M(\boldsymbol{\theta})
\end{equation}
for a suitable Schwartz seminorm $p_M$.

\begin{proof}[Completion of the proof of Proposition \ref{prop:complete}]
We consider the proof of claim (i) and first show the embedding $\bigoplus_{j=1}^N\dS\hookrightarrow B^{s}_{p,q}(A,W)$. Let $\vf\in \bigoplus_{j=1}^N\dS$, and put  $w(x):=\|W^{1/p}(x)\|^p$, which by Lemma \ref{le:do} is in the scalar $A_{\max\{1,p\}}(\bR^d)$. We will need the fact that a scalar $A_p$-weight has moderate average growth in the sense that
\begin{equation}\label{eq:gr}
    \int_{\bR^d} w(x) \brac{x}_A^{-(\nu \max\{1,p\}+\epsilon)}\,dx<+\infty,
    \end{equation}
   for any $\epsilon>0$, which follows from the observation in  \eqref{eq:doubb2} and Lemma \ref{le:sq}.

 We have, for $m> \nu\max\{1,p\}$,

\begin{align}
\int_{\bR^d} |W^{1/p}(x)	\phi_k(D)\vf(x)|^p\dx&\leq 
\int_{\bR^d} \|W^{1/p}(x)\|^p|	\phi_k(D)\vf(x)|^p\dx\nonumber\\
&= \int_{\bR^d} w(x)|	\phi_k(D)\vf(x)|^p\dx\nonumber\\
&\leq  \|\brac{\cdot}_A^{m}	\phi_k(D)\vf\|_\infty^p\int_{\bR^d} w(x)\brac{x}_A^{-m}\dx\nonumber\\
&\leq  C\|\brac{\cdot}_A^{m}	\phi_k(D)\vf\|_\infty^p,\label{eq:wp}
\end{align}
where we used \eqref{eq:gr}.  For the scalar function $f_i\in \dS$ it can be shown (see, e.g., \cite[Appendix A]{BoruNiel:2008a}) that for $m>0$, $r>0$, and $N>\max\{m,(\nu+1)/\alpha_2+r\} $, 
\begin{equation}\label{eq:thet}
\|\brac{\cdot}_A^m	\phi_k(D)f_i\|_\infty\leq ct_k^{-r}p_N(f_i),
\end{equation}
with $c$ independent of $f_i$. 
Hence, by \eqref{eq:wp}, we obtain
$$\|	\phi_k(D)\vf\|_{L^p(W)}\leq c't_k^{-r}p_N(\vf).$$ Recall that we may choose $r$ arbitrarily large by increasing $N$, and for sufficiently large $r$, we obtain
$$\|\vf\|_{B^{s}_{p,q}(A,W)}=\|\{t_k^s\|	\phi_k(D)\vf\|_{L^p(W)}\}_k\|_{\ell^q}\leq c'p_N(\vf),$$
 for $c'$ independent of $\vf$, where we used that $\sum_{k} t_k^{-\nu-\epsilon}<\infty$ for $\epsilon>0$, which is an easy consequence of \eqref{eq:integral}. This provides the wanted embedding.
 
 We now turn to the embedding $B^{s}_{p,q}(A,W)\hookrightarrow \bigoplus_{j=1}^N\mathcal{S}'(\bR^d)$. Let us first consider the case $1<p<\infty$.
  Take $\vf\in B^{s}_{p,q}(A,W)$ and let $\boldsymbol{\theta}=(\theta_1,\ldots,\theta_N)^\top\in  \bigoplus_{j=1}^N\dS$. Then, using the smooth resolution of the identity $\sum_k \psi_k^2(\xi)=1$ from Equation~\eqref{eq:bapusq}, we obtain
 \begin{align*}
     \langle \vf,\boldsymbol{\theta}\rangle_{\bC^N}
     &=\sum_{k\in\bN}  \langle \psi_k(D)\vf,\psi_k(D)\boldsymbol{\theta}\rangle_{\bC^N}\\
     &=\sum_{k\in\bN}  \langle t_k^{s}W^{1/p}\psi_k(D)\vf,t_k^{-s}W^{-1/p}\psi_k(D)\boldsymbol{\theta}\rangle_{\bC^N}.
 \end{align*}
 In the case $1\leq q\leq \infty$, we use H\"older's inequality twice to obtain
 $$\int_{\bR^d} |\langle \vf(x),\boldsymbol{\theta}(x)\rangle_{\bC^N}|\dx
 \leq
 \|\{t_k^s\|\psi_k(D)\vf\|_{L^p(W)}\}_k\|_{\ell^q}
 \|\{t_k^{-s}\|\psi_k(D)\boldsymbol{\theta}\|_{L^{p'}(W^{-p'/p})}\}_k\|_{\ell^{q'}},
 $$
 with $q'$ the dual H\"older exponent to $q$.
 It holds that $W^{-p'/p}\in \mathbf{A}_{p'}$, cf.\ Remark \ref{re:dual}. Hence, we may use the embedding already obtained to conclude that, with $N$ suitably large, $$\|\{t_k^{-s}\|\phi_k(D)\boldsymbol{\theta}\|_{L^{p'}(W^{-p'/p})}\}_k\|_{\ell^{q'}}\leq c p_N(\boldsymbol{\theta}),$$ so
 $$\int_{\bR^d} |\langle \vf(x),\boldsymbol{\theta}(x)\rangle_{\bC^N}|\dx
 \leq c\|\vf\|_{B^{s}_{p,q}(A,W)}p_N(\boldsymbol{\theta}),$$
 which proves the continuous embedding $B^{s}_{p,q}(A,W)\hookrightarrow \bigoplus_{j=1}^N\mathcal{S}'(\bR^d)$. For $0<q<1$, we use the embedding estimate $\|\cdot\|_{{B^s_{p,\infty}(A,W)}}\leq \|\cdot\|_{{B^s_{p,q}(A,W)}}$, which follows directly from the inclusion properties of the $\ell^q$-spaces. Then we proceed using the estimate
 $$\int_{\bR^d}|\langle \vf,\boldsymbol{\theta}\rangle_{\bC^N}|\,\dx  \leq \|\{t_k^s\|\psi_k(D)\vf\|_{L^p(W)}\}_k\|_{\ell^\infty}
 \|\{t_k^{-s}\|\psi_k(D)\boldsymbol{\theta}\|_{L^{p'}(W^{-p'/p})}\}_k\|_{\ell^{1}}.$$
 The preceding estimate also proves the assertion for $0<q<1$, since the second factor is bounded by the embedding already obtained with the $\ell^1$-norm in the sequence variable.

 It remains to consider the case $0<p\leq 1$, where we estimate the distributional pairing directly by means of the frame coefficients and the scalar lower-average estimate \eqref{eq:scalar-lower-average}. Let
\[
\vc_{k,\ell}:=\langle \vf,\omega_{k,\ell}\rangle,
\qquad
\vd_{k,\ell}:=\langle \boldsymbol{\theta},\omega_{k,\ell}\rangle .
\]
The tight-frame expansion of $\boldsymbol{\theta}$  also converges in
$\bigoplus_{j=1}^N\mathcal S(\mathbb R^d)$: indeed, the rapid decay of
the coefficients $\vd_{k,\ell}=\langle\boldsymbol{\theta},\omega_{k,\ell}\rangle$,
together with the uniform Schwartz estimates for the frame elements
$\omega_{k,\ell}$, gives unconditional convergence in every Schwartz
seminorm; tightness then identifies the limit with $\boldsymbol{\theta}$.
 Hence, we obtain
\[
\langle \vf,\boldsymbol{\theta}\rangle
=
\sum_{k\in\bN}\sum_{\ell\in\bZ^d}
\langle \vc_{k,\ell},\vd_{k,\ell}\rangle_{\bC^N}.
\]
By \eqref{eq:scalar-lower-average},
\[
|\vc_{k,\ell}|\leq
C t_k^{M_0}(1+|x_{k,\ell}|_A)^{M_0}|U(k,\ell)|^{-1/p}
\bigg(\int_{U(k,\ell)} |W^{1/p}(x)\vc_{k,\ell}|^p\,\d x\bigg)^{1/p}.
\]
Hence, using $0<p\leq 1$ and the elementary inequality
$\sum_\ell a_\ell\leq (\sum_\ell a_\ell^p)^{1/p}$ for non-negative sequences, we obtain for each fixed $k$,
\begin{align*}
\sum_{\ell\in\bZ^d}
\big|\langle \vc_{k,\ell},\vd_{k,\ell}\rangle_{\bC^N}\big|
&\leq
C\sup_{\ell\in\bZ^d}
\bigg[t_k^{M_0}(1+|x_{k,\ell}|_A)^{M_0}|\vd_{k,\ell}|\bigg] \\
&\quad \times
\bigg(
\sum_{\ell\in\bZ^d}|U(k,\ell)|^{-1}
\int_{U(k,\ell)} |W^{1/p}(x)\vc_{k,\ell}|^p\,\d x
\bigg)^{1/p}.
\end{align*}
By the definition of the coefficient quasi-norm and the estimate $|U(k,\ell)|\asymp t_k^{-\nu}$,
\[
\bigg(
\sum_{\ell\in\bZ^d}|U(k,\ell)|^{-1}
\int_{U(k,\ell)} |W^{1/p}(x)\vc_{k,\ell}|^p\,\d x
\bigg)^{1/p}
\lesssim
t_k^{\nu(1/p-1/2)-s}
\|\{\vc_{k,\ell}\}_{k,\ell}\|_{b^s_{p,\infty}(A,W)} .
\]
Consequently,
\begin{align*}
|\langle \vf,\boldsymbol{\theta}\rangle|
&\lesssim
\|\{\vc_{k,\ell}\}_{k,\ell}\|_{b^s_{p,\infty}(A,W)}
\sum_{k\in\bN}t_k^{\nu(1/p-1/2)-s}
\sup_{\ell\in\bZ^d}
\big[t_k^{M_0}(1+|x_{k,\ell}|_A)^{M_0}
\big|\langle \boldsymbol{\theta},\omega_{k,\ell}\rangle\big|\big].
\end{align*}
By Proposition \ref{prop:coef}, applied with $q=\infty$, and the embedding
$B^s_{p,q}(A,W)\hookrightarrow B^s_{p,\infty}(A,W)$, we obtain
\[
\|\{\vc_{k,\ell}\}_{k,\ell}\|_{b^s_{p,\infty}(A,W)}
\lesssim
\|\vf\|_{B^s_{p,\infty}(A,W)}
\leq
\|\vf\|_{B^s_{p,q}(A,W)} .
\]
The remaining factor is bounded by \eqref{eq:theta-weighted-summability}. Therefore, for a sufficiently large Schwartz seminorm $p_M$,
\[
|\langle \vf,\boldsymbol{\theta}\rangle|
\leq
C\|\vf\|_{B^s_{p,q}(A,W)}p_M(\boldsymbol{\theta}).
\]
This proves the continuous embedding
\[
B^s_{p,q}(A,W)\hookrightarrow
\bigoplus_{j=1}^N\mathcal{S}'(\bR^d)
\]
in the case $0<p\leq 1$, and completes the proof of part (i).

We turn to the proof of part (ii). Suppose that $\{\vf_n\}_n$ is a Cauchy sequence in $B^{s}_{p,q}(A,W)$. The sequence $\{\vf_n\}_n$ is also Cauchy in the complete space $\bigoplus_{j=1}^N\mathcal{S}'(\bR^d)$ by part (i). Hence, the sequence has a well-defined limit $\vf\in \bigoplus_{j=1}^N\mathcal{S}'(\bR^d)$. Recall that we have $\phi_j\in\dS$ so it follows that $\lim_{m\rightarrow \infty}[\phi_j(D)\vf_m](x)= [\phi_j(D)\vf](x)$ pointwise for $x\in\bR^d$.
Hence, by an iterated application of Fatou's lemma,
\begin{align*}
    \|\vf-\vf_n\|_{B^{s}_{p,q}(A,W)}&\asymp
    \bigg(\sum_{j=1}^\infty t_j^{sq}\|\phi_j(D)(\vf-\vf_n)\|_{L^p(W)}^q\bigg)^{1/q}\\
    &\leq \bigg(\sum_{j=1}^\infty t_j^{sq}\liminf_{m\rightarrow \infty}\|\phi_j(D)(\vf_m-\vf_n)\|_{L^p(W)}^q\bigg)^{1/q}\\
    &\leq \liminf_{m\rightarrow \infty}\bigg(\sum_{j=1}^\infty t_j^{sq}\|\phi_j(D)(\vf_m-\vf_n)\|_{L^p(W)}^q\bigg)^{1/q}\\
    &\asymp \liminf_{m\rightarrow \infty} \|\vf_m- \vf_n\|_{B^{s}_{p,q}(A,W)}<+\infty.
\end{align*}
 We first deduce from the estimate that $\vf=(\vf-\vf_n)+\vf_n\in B^{s}_{p,q}(A,W)$. Moreover,  
it also follows from the same estimate that 
$\vf_n\rightarrow \vf$ in $B^{s}_{p,q}(A,W)$ by using that
$\{\vf_n\}_n$ is a Cauchy sequence in $B^{s}_{p,q}(A,W)$. This proves completeness of $B^{s}_{p,q}(A,W)$ and concludes the proof.
\end{proof}


\begin{thebibliography}{10}

\bibitem{BoruNiel:2008a}
L.~Borup and M.~Nielsen.
\newblock On anisotropic {T}riebel-{L}izorkin type spaces, with applications to the study of pseudo-differential operators.
\newblock {\em J. Funct. Spaces Appl.}, 6(2):107--154, 2008.

\bibitem{BuHytoYang:2025a}
F.~Bu, T.~Hyt\"onen, D.~Yang, and W.~Yuan.
\newblock Matrix-weighted {B}esov-type and {T}riebel-{L}izorkin-type spaces {I}: {$A_p$}-dimensions of matrix weights and {$\psi$}-transform characterizations.
\newblock {\em Math. Ann.}, 391(4):6105--6185, 2025.

\bibitem{Cald:1976a}
A.~Calder{\'o}n.
\newblock Inequalities for the maximal function relative to a metric.
\newblock {\em Studia Mathematica}, 57(3):297--306, 1976.

\bibitem{ChriGold:2001a}
M.~Christ and M.~Goldberg.
\newblock Vector {$A_2$} weights and a {H}ardy-{L}ittlewood maximal function.
\newblock {\em Trans. Amer. Math. Soc.}, 353(5):1995--2002, 2001.

\bibitem{CruzMoenRodn:2016a}
D.~Cruz-Uribe, K.~Moen, and S.~Rodney.
\newblock Matrix {$ A_p$} weights, degenerate {S}obolev spaces, and mappings of finite distortion.
\newblock {\em J. Geom. Anal.}, 26(4):2797--2830, 2016.

\bibitem{FrazJaweWeis:1991a}
M.~Frazier, B.~Jawerth, and G.~Weiss.
\newblock {\em Littlewood-{P}aley theory and the study of function spaces}, volume~79 of {\em CBMS Regional Conference Series in Mathematics}.
\newblock Conference Board of the Mathematical Sciences, Washington, DC; by the American Mathematical Society, Providence, RI, 1991.

\bibitem{FrazRoud:2004a}
M.~Frazier and S.~Roudenko.
\newblock Matrix-weighted {Besov} spaces and conditions of ${A}_p$ type for $0<p\leq 1$.
\newblock {\em Indiana Univ. Math. J.}, 53(5):1225--1254, 2004.

\bibitem{GeorJohnNiel:2017a}
A.~G. Georgiadis, J.~Johnsen, and M.~Nielsen.
\newblock Wavelet transforms for homogeneous mixed-norm {T}riebel-{L}izorkin spaces.
\newblock {\em Monatsh. Math.}, 183(4):587--624, 2017.

\bibitem{Gold:2003a}
M.~Goldberg.
\newblock Matrix {$A_p$} weights via maximal functions.
\newblock {\em Pacific J. Math.}, 211(2):201--220, 2003.

\bibitem{HuntMuckWhee:1973a}
R.~Hunt, B.~Muckenhoupt, and R.~Wheeden.
\newblock Weighted norm inequalities for the conjugate function and {H}ilbert transform.
\newblock {\em Trans. Amer. Math. Soc.}, 176:227--251, 1973.

\bibitem{IsraKwonPott:2017a}
J.~Isralowitz, H.-K. Kwon, and S.~Pott.
\newblock Matrix weighted norm inequalities for commutators and paraproducts with matrix symbols.
\newblock {\em J. Lond. Math. Soc. (2)}, 96(1):243--270, 2017.

\bibitem{Jawe:1986a}
B.~Jawerth.
\newblock Weighted inequalities for maximal operators: linearization, localization and factorization.
\newblock {\em Amer. J. Math.}, 108(2):361--414, 1986.

\bibitem{John:2014a}
F.~John.
\newblock Extremum problems with inequalities as subsidiary conditions.
\newblock In G.~Giorgi and T.~H. Kjeldsen, editors, {\em Traces and Emergence of Nonlinear Programming}, pages 197--215. Springer Basel, Basel, 2014.

\bibitem{Kurt:1975a}
D.~S. Kurtz.
\newblock Weighted norm inequalities for the {H}ardy-{L}ittlewood maximal function for one-parameter rectangles.
\newblock {\em Studia Math.}, 53(1):39--54, 1975.

\bibitem{NazaTrei:1996a}
F.~L. Nazarov and S.~R. Tre\u{\i}l\cprime.
\newblock The hunt for a {B}ellman function: applications to estimates for singular integral operators and to other classical problems of harmonic analysis.
\newblock {\em Algebra i Analiz}, 8(5):32--162, 1996.

\bibitem{Niel:2025a}
M.~Nielsen.
\newblock Bandlimited multipliers on matrix-weighted {$L^p$}-spaces.
\newblock {\em J. Fourier Anal. Appl.}, 31(1):Paper No. 3, 10, 2025.

\bibitem{NielRasm:2018a}
M.~Nielsen and M.~G. Rasmussen.
\newblock Projection operators on matrix weighted {$L^p$} and a simple sufficient {M}uckenhoupt condition.
\newblock {\em Math. Scand.}, 123(1):72--84, 2018.

\bibitem{NielSiki:2025a}
M.~Nielsen and H.~\v{S}iki\'c.
\newblock Muckenhoupt matrix weights for general bases.
\newblock In {\em The mathematical heritage of {G}uido {W}eiss}, Appl. Numer. Harmon. Anal., pages 361--386. Birkh\"auser/Springer, Cham, 2025.

\bibitem{Pere:1991b}
C.~P\'erez.
\newblock Weighted norm inequalities for general maximal operators.
\newblock {\em Publ. Mat.}, 35(1):169--186, 1991.
\newblock Conference on Mathematical Analysis (El Escorial, 1989).

\bibitem{RiccStei:1987a}
F.~Ricci and E.~M. Stein.
\newblock Harmonic analysis on nilpotent groups and singular integrals. {I}. {O}scillatory integrals.
\newblock {\em J. Funct. Anal.}, 73(1):179--194, 1987.

\bibitem{Roud:2003a}
S.~Roudenko.
\newblock Matrix-weighted {B}esov spaces.
\newblock {\em Trans. Amer. Math. Soc.}, 355(1):273--314, 2003.

\bibitem{Stei:1993a}
E.~M. Stein.
\newblock {\em Harmonic analysis: real-variable methods, orthogonality, and oscillatory integrals}, volume~43 of {\em Princeton Mathematical Series}.
\newblock Princeton University Press, Princeton, NJ, 1993.

\bibitem{SteiWain:1978a}
E.~M. Stein and S.~Wainger.
\newblock Problems in harmonic analysis related to curvature.
\newblock {\em Bull. Amer. Math. Soc.}, 84(6):1239--1295, 1978.

\bibitem{TreiVolb:1997a}
S.~Treil and A.~Volberg.
\newblock Wavelets and the angle between past and future.
\newblock {\em J. Funct. Anal.}, 143(2):269--308, 1997.

\bibitem{Volb:1997a}
A.~Volberg.
\newblock Matrix {$A_p$} weights via {$S$}-functions.
\newblock {\em J. Amer. Math. Soc.}, 10(2):445--466, 1997.

\end{thebibliography}
\end{document}